\documentclass[11pt,
]{amsart}
\usepackage{
amssymb, graphicx
}

\usepackage{
xcolor,
cite%,refcheck
}
 
\date{\today}
\usepackage{datetime}
\numberwithin{equation}{section}

\newtheorem{theorem}{Theorem}[section]
\newtheorem{proposition}{Proposition}[section]
\newtheorem{lemma}{Lemma}[section]

\theoremstyle{definition}
\newtheorem{remark}{Remark}[section]

\DeclareMathOperator{\supp}{supp}

\DeclareMathOperator{\sign}{sign}
\newcommand{\eps}{\varepsilon}

\newcommand{\R}{\mathbf{R}}

\newcommand{\Id}{\textrm{\rm Id}}
\renewcommand{\r}[1]{(\ref{#1})}

\newcommand{\be}[1]{\begin{equation}\label{#1}}
\newcommand{\ee}{\end{equation}}

\renewcommand{\sc}{semi\-classical}
\newcommand{\p}{{\partial}}

\renewcommand{\d}{\mathrm{d}}

\renewcommand{\i}{\mathrm{i}}

\newcommand{\bo}{\partial \Omega}

\renewcommand{\a}{\mathsf{a}}

\title[Recovery of a cubic non-linearity]{Recovery of a cubic non-linearity in the wave equation in the weakly non-linear regime}

\author[A. S\'a Barreto]{Ant\^ onio S\'a Barreto}
\author[P. Stefanov]{Plamen Stefanov}
\address{Department of Mathematics, Purdue University, West Lafayette, IN 47907}
\thanks{The first author is partly supported by the Simons Foundation grants \#349507 and  \#848410. 
The second author is partly supported by the  NSF Grant DMS-1900475.}

\begin{document}
 \begin{abstract}
We study the inverse problem of recovery a compactly supported non-linearity in the semilinear wave equation $u_{tt}-\Delta u+ \alpha(x) |u|^2u=0$, in two and three dimensions. We probe the medium with complex-valued harmonic waves of wavelength $h$ and amplitude $h^{-1/2}$,  then they propagate in the weakly non-linear regime; and measure the transmitted wave when it exits $\supp\alpha$. We show that one can extract the Radon transform of $\alpha$ from the phase shift of such waves, and then one can recover $\alpha$. We also show that one can probe the medium with real-valued  harmonic waves and obtain uniqueness for the linearized problem. 
\end{abstract} 
\maketitle

\section{Introduction}  
Consider the semilinear wave equation
 \be{1}
u_{tt}-\Delta u+ \alpha(x) |u|^2u=0,\quad (t,x)\in \R_t\times\R_x^n,
\ee
where $0\le \alpha\in  C_0^\infty$, which corresponds to the ``defocusing" case. 
%I have seen it written this way rather than with the $u^3$ non-linearity. One of the technical advantages here is that the non-linearity ``commutes'' with $e^{\i \phi}$, $\phi$ real. 
There is energy preserved under the dynamics
\be{1a}
E(\mathbf{u}(t)):=\int\left(\frac12 |u_t|^2+\frac12|\p_x u|^2  +\frac14\alpha(x)|u|^4    \right) \d x= \text{const.},
\ee
where $\mathbf{u}=(u,u_t)$. Assume that $\supp\alpha\subset B(0,R)$, where $B(0,R)=\{x\in \R^n;\; |x|<R\}$, $n\ge2$. We send high-frequency waves from the exterior of $B(0,R)$, wait until their high frequency part exits $B(0,R)$, and measure them. We make this more precise below. The problem we study is if we can recover $\alpha$ given that information, and how to do it.

%[history of the problem here]
It is known that one can recover such a non-linearity and even more general ones, see, e.g., \cite{LassasUW_2016, KLU-18, OSSU-principal}. Since this is an inverse problem for a non-linear PDE, one would expect that the information about $\alpha$ would be encoded in the data in many ways. We refer to section~\ref{sec_comparison} for a comparison with the existing approaches. The standard approach is to collide several small waves moving in the linear regime, which meet at a point in time-space and produce and even smaller signal (among the rest) with a weak singularity which can be used to recover the non-linearity at the collision point. 
The novelty of this work is that we want to use waves (solutions) that are not too small and do not have very weak singularities. In fact, the solutions we use will have amplitudes and  energies increasing with the frequency. We chose the balance between the amplitude and frequency so that the waves propagate in the weakly non-linear regime. As we see below, for \r{1} this means an amplitude $\sim h^{-1/2}$ when the frequency is $\sim h^{-1}$, $0<h\ll1$. Such large amplitude solutions have better chance to be measured reliably in presence of background noise, and the non-linearity starts to affect the propagation of the waves. 
 Moreover, we show that $\alpha$ can be reliably recovered from the principal part of the outgoing signal, having an amplitude of the same order $\sim h^{-1/2}$ rather than from lower order terms. 
The mathematical theory of weakly non-linear waves for semi- and quasilinear first order hyperbolic systems has been developed in \cite{Metivier-Notes,Metivier-Joly-Rauch, Joly-Rauch_just, Donnat-Rauch_dispersive, Dumas_Nonlinear-Geom-Optics, JMR-95},  and is known in the physics literature as well, see, e.g., \cite{DajaniII1990Gbpi, dajani1993weakly}. Solvability of semilinear wave equations has been studied in \cite{Heinz-Wahl1975, Brenner79, Segal63, Reed-Simon2, dodson2018global,  Jorgens1961,  Ebihara72} and other works. In particular, \r{1} has a global solution with initial conditions with finite local energy at least in dimensions $n=2,3$, continuously depending on the initial data and sources, see Appendix~\ref{sec_gl}, and has a local one in all dimensions regardless of the sign of $\alpha$. 

To make our setup more precise, let $0\le \chi\in C_0^\infty$ be such that  $\supp\chi\subset (-\delta,\delta)$ with some $\delta>0$. 
We probe the medium with the incoming wave 
\be{10a}
u_\textrm{in}= e^{\i (-t+x\cdot\omega)/h}h^{-1/2}\chi(-t+x\cdot\omega), \quad \omega\in S^{n-1}. 
\ee
When $t<-R-\delta$, the wave is outside $B(0,R)\supset\supp \alpha$, solves  the linear wave equation and is moving in the direction $\omega$. We solve \r{1} with initial condition
\be{1c}
u=u_\text{in}\quad \text{for $t<-R-\delta$}.
\ee
We measure 
\be{2a}
\Lambda(u_\text{in}) = u|_{t=T, \; |x\cdot\omega-T|\le\delta},
\ee
where $T>R+\delta$ is fixed. A visual representation of $u_\text{in}$ and $\Lambda(u_\text{in})$ can be found in Figures~\ref{fig_K1}, \ref{fig_K5} and \ref{fig_K5_h005}, left and center, respectively. 
If $\alpha=0$, we would have $\Lambda(u_\text{in})= e^{\i (-T+x\cdot\omega)/h}h^{-1/2}\chi(-T+x\cdot\omega)$, and the information about $\alpha$ in the general case would be encoded in the deviation of $\Lambda(u_\text{in})$ from that. We do that for every unit $\omega$.  
It turns out that the non-linearity in the weakly non-linear regime does not change the geometry of the propagation of the \sc\ wave front set (with this particular incoming wave), i.e., it  propagates ``in a linear way'' as a set, but changes the leading order amplitude denoted by $a_0$ below. The modulus $|a_0|$ remains unaffected by the non-linearity but its argument depends on $\alpha$ through its X-ray transform $X\alpha$ of $\alpha$ along lines parallel to $\omega$. The non-linearity creates a phase shift of lower order compared to $(-t+x\cdot\omega)/h$ (no $h$ in it).  We can extract $X\alpha$ from the data in an explicit way, and then we can recover $\alpha$. We present numerical simulations as well. 

Note that the form of the PDE \r{1} and the choice of the probing waves \r{1c} prevent formation of harmonics which makes a parametrix of the type \r{2} below possible. In general, one would expect solutions of the kind $u = U(\phi/h,t, x,h)$, see also section~\ref{sec_real}. The propagating wave interacts with itself though, at principal amplitude level, which explains the phase shift. 

Our main result is the following. 

\begin{theorem}\label{thm_main}
Let $n=2$ or $n=3$,  let $u_{\textrm{\rm in}}$ be as in \r{10a}, and let $\alpha\in C_0^\infty(\R^n;\;\R)$.  Then 
\be{eq:main_thm}
h^{1/2}e^{\i(T-x\cdot\omega)/h} \Lambda(u_\textrm{\rm in}   ) = \chi(-T+x\cdot \omega)\exp\Big(-\frac{\i}2 \chi^2(-T+x\cdot \omega).X\alpha\Big) + O(h),
\ee
holds in the uniform norm for $(x,\omega)$, where $X\alpha=X\alpha(x,\omega) = \int \alpha(x+s\omega)\, \d s$ %,  $z=x- (x\cdot \theta)\theta\in\omega^\perp$ 
is the X-ray transform of $\alpha$. 
\end{theorem}

This recovers $X\alpha$; we give more details in section~\ref{sec_ps} without formulating the analysis there as a theorem. Knowing $X\alpha$, we can recover $\alpha$ as well.

We also study the same problem with a real incoming wave
\[
u_\textrm{in}= \cos((-t+x\cdot\omega)/h)h^{-1/2}\chi(-t+x\cdot\omega), \quad \omega\in S^{n-1}.
\]
Then the solution remains real, and we can write the non-linearity as $\alpha u^3$. In this case, harmonics do develop, and aside from the leading frequency $1/h$, we also get $k/h$ with $k$ odd. The transport equations do not seem to be solvable in an explicit way but they are still solvable, see Appendix~\ref{sec_transp}. In section~\ref{sec_real}, we show that for $\alpha\ll1$, one can linearize the data and solve the linearized problem to recover $\alpha$ approximately. 

\textbf{Acknowledgments.}  The authors want to thank Jason Murphy for pointing out that certain integrals had incorrect limits in the first version; and the anonymous referees for their helpful suggestions. 

\section{Summary of the three regimes in non-linear geometrical optics} \label{sec_2}

We are looking for an asymptotic solution of \r{1} of the form
\be{2}
u=e^{\i\phi(t,x,\omega)/h}h^p a(t,x,\omega,h)  %e^{\i(-t+x\cdot\omega)/h}
\ee
with $0\le h\ll1$, $|\omega|=1$, and $a$ having some asymptotic expansion in powers of $h$, not necessarily integer ones a priori. We normalize the expansion by requiring $a=a_0(t,x)+\dots$ with $a_0$ independent of $h$. More generally, one can consider $u$ of the kind 
\be{2aa}
u= e^{\i \phi/h}\eps a, \quad a\sim 1,
\ee
and work in regimes $h^{k_1}/C\le \eps\le C  h^{k_2}$.

Plug \r{2} into \r{1} to get
\be{3}
\begin{split}
e^{-\i\phi /h}P(e^{\i\phi /h}h^p a) &= h^pe^{-\i\phi /h}\Box (e^{\i\phi /h} a) +  h^{3p}\alpha(x)|a|^2 a \\
&= h^{p-2}(-\phi_t^2+|\p_x\phi|^2)a+ 2\i h^{p-1}(\phi_t, -\p_x\phi)\cdot(\p_t, \p_x)a\\ &\quad +\i h^{p-1} (\Box\phi )a + h^p\Box a+  h^{3p}\alpha(x) |a|^2a\\&=0,
\end{split}
\ee
where $P(u)$ is the non-linear operator on the l.h.s.\ of \r{1}.

 One studies solutions with various $p$, and depending on $p$, one has different asymptotics, see \cite{Metivier-Notes}. 
 Three asymptotic regimes are distinguished there.  

\subsection{The linear regime} \label{sec_l}
We want the phase $\phi$ to be unaffected by the non-linearity and we want the principal part of the amplitude to be the same as in the linear case. For the first requirement, we need $p-2<3p$, and for the second one: $p-1<3p$. In other words, $p>-1/2$. If we want an expansion of $a$ in integer powers of $h$ (which provides a good estimate of the error, i.e., for the non-linear part), we need $p\ge0$. %So we take $p=0$. 
This corresponds to $\eps=1$ in \r{2aa}. 

Assume $p>-1/2$ below. 
Then we get the eikonal equation
\be{5}
\phi_t^2-|\p_x\phi|^2=0,
\ee
and the first transport equation is the same as in the linear case:
\be{6}
Ta_0=0, \quad T:= 2 (\phi_t, -\p_x\phi)\cdot (\p_t, \p_x) + \Box\phi. 
\ee
The easiest to understand case is $p=0$; then the non-linear effect is of order $h$. Note that when $p=0$, the solution is not small; the amplitude is $\sim 1$. When $p\in (-1/2,0]$, the solution increases when $h\to0$ (in the $L^2$ norm).   

This case is somewhat similar to the propagation of conormal singularities for one progressing wave \cite{Bony1}, which states that if the initial data is conormal to a characteristic hypersurface $\Sigma$,   and it is smooth enough, then the solution of \r{1} remains conormal to $\Sigma$ and its principal symbol satisfies \r{6}.  In this case, the subprincipal symbol of the solution involves the $X$-ray transform of $\alpha(x) u|u|^2$,  which is not enough to determine $\alpha(x)$. 

\subsection{The weakly non-linear regime} \label{sec_wnl}
This happens when the eikonal equation stays the same but the first transport equation involves the non-linearity. To have this, we need $p-1=3p$, see \r{3}, i.e., $p=-1/2$. Then $u\sim h^{-1/2}\to\infty$. The first transport equation then takes the form
\be{7}
Ta_0-\i \alpha(x)|a_0|^2a_0=0 .
\ee
We analyze this case further in next section.

\subsection{The fully non-linear regime} Now we want $p$ to be such that the eikonal equation gets modified. Then $p-2=3p$, i.e., $p=-1$.  The ansatz described in \cite{Metivier-Notes} is to look for a solution of the form
\[
u = h^{-1}U(t,x,\phi/h), \quad U(t,x,\theta)\sim U_0+hU_1+\dots
\]
with $U_k$ periodic in $\theta$. Set $\sigma^2=\phi_t^2-|\p_x\phi|^2$; then
\be{13}
\begin{split}
\sigma^2\p_\theta^2U_0+ |U_0|^2U_0&=0,\\
\sigma^2\p_\theta^2U_1 +   \left(2\alpha U_0^2\bar U_1+ 2 \alpha|U_0|^2 U_1\right) + T\p_\theta U_0 &=0,
\end{split}
\ee
etc. 

If we view $\sigma$ as an independent variable in the ``eikonal equation'', we seek $U_0$ of the form $U_0 = \sigma G$; then $G=e^{\i\theta}$ works (also, the conjugate does). One then guesses that the first one is enough, and one gets
\[
U_0= C(t,x) \sigma e^{\i \phi/h}. 
\]
We still have to determine $\phi$, and therefore, $\sigma$.  The phase $\phi$ now solves a quasi-linear hyperbolic PDE of order two instead of the eikonal equation. More details can be found in   \cite[p.~85]{Metivier-Notes} and the references there. We are not going to consider this regime in the present work.

\section{Recovery of the non-linearity through phase shifts} \label{sec_ps}
\subsection{Analysis of the principal term} 
We analyze the weakly non-linear regime in more detail now and in particular, we prove Theorem~\ref{thm_main}.  
Let us take a linear phase: $\phi = -t+x\cdot\omega$, $|\omega|=1$ as in \r{10a}. Then \r{7} takes the form 
\be{8}
2 (1, \omega)\cdot (\p_t, \p_x)a_0 +\i \alpha(x)|a_0|^2a_0=0.
\ee
Along the characteristics, parameterized properly by some $s$ as in \r{11} below, we get
\be{9}
2\frac{\d}{\d s}a_0+\i \alpha|a_0|^2a_0=0
\ee
with $\alpha=\alpha(x(s))$, which we denote by $\alpha(s)$ below. We have
\be{aoc}
2\frac{\d}{\d s}|a_0|^2= \dot a_0 \bar a_0+  a_0 \dot {\bar a}_0 = -4\Re(\i \alpha |a_0|^4 )=0.
\ee
Therefore, $|a_0|=\text{const}$ along any fixed characteristic, call it $A$ (the physical intensity is $A^2$), as in the analysis of the Maxwell-Bloch system in \cite[p.69]{Metivier-Notes}. Now we are solving $2\dot a_0+\i A^2\alpha a_0=0$, which has the solution
\[
a_0 = C\exp\Big(-\i\frac{ A^2}2\int \alpha(s)\,\d s\Big). 
\]
Since $|a_0|=A$, the phase of $C$ can be added to the integral above, which would give us another indefinite integral; thus
\be{10}
a_0 = A\exp\Big(-\i \frac{ A^2}2 \int \alpha(s)\,\d s\Big). 
\ee
One can see directly that this solves \r{9}. 

Recall that we have the incoming wave \r{10a} and assume $\chi\ge0$. In time-space, for a fixed $\omega$, introduce the variables $(s,y)=(t,x-t\omega)$; then $(t,x)=(s,y+s\omega)$. Then $\p_s= \p_t+\omega\cdot\p_x$ which justifies \r{9}. The initial condition would be $a_0=\chi(y\cdot\omega)$ for $s\ll0$. Then we can take a definite integral in \r{10} from $-\infty$ to $s$, with a ``dummy variable'' of integration $\sigma$; then 
$A=\chi(y\cdot\omega)$. Passing to the $(t,x)$ variables, after replacing $t-\sigma$ in the integrand by $s$, we get
\be{11}
a_0(t,x) = A\exp\Big(-\i \frac{ A^2}2 \int_{-\infty}^0 \alpha(x+s\omega)\,\d s\Big), \quad A= \chi(-t+x\cdot\omega). 
\ee
An approximate solution (we will justify this later) is
\be{12}
\tilde u = h^{-1/2}\chi(-t+x\cdot\omega) e^{i\Phi}, \quad \Phi:=\frac{-t+x\cdot\omega}{h} -\frac12  \chi^2(-t+x\cdot\omega) \int_{-\infty}^0 \alpha(x+s\omega)\,\d s.
\ee
We  show later that $ u= \tilde u+O(h)$.  %$x\cdot\omega=R\gg1$ 
With \r{2a} in mind, we take $t=T$ to get
\[
h^{1/2}e^{\i(t-x\cdot\omega)/h}u|_{t=T, \, |x\cdot\omega-T|\le\delta} = \chi(-T+x\cdot \omega)\exp\Big(-\frac{\i}2 \chi^2(-T+x\cdot \omega).X\alpha\Big) + O(h),
\]
where $X\alpha = X\alpha(x-(x\cdot\omega)\omega,\omega)$ is the X-ray transform of $\alpha$ along the ray $x\mapsto x+s\omega$, written in coordinates $(z,\omega)$, $z\in\omega^\perp$. By \r{2a} again, we want to restrict this to $|x\cdot\omega-T|\le\delta$ but this is already done since the support of the $\chi$ term lies there (but maybe  not the $O(h)$ error). 
% setting $K:= \chi(0)$, we get
Hence,
\[
h^{1/2}e^{\i(T-x\cdot\omega)/h} \Lambda(u_\textrm{in}) = \chi(-T+x\cdot \omega)\exp\Big(-\frac{\i}2 \chi^2(-T+x\cdot \omega).X\alpha\Big) + O(h). 
\]
This is equality \r{eq:main_thm} in the theorem. 
This  data is over determined and one can restrict it to $x\cdot\omega=T$, for example, to get 
\be{13a}
h^{1/2} \Lambda(u_\textrm{in})|_{x\cdot\omega=T} = K\exp\Big(-\i\frac{K^2}2 X\alpha\Big)+O(h), \quad K:=\chi(0). 
\ee
It is easy to see now that we get $X\alpha$ modulo $4\pi/K^2$ up to an $O(h)$ error. If we have an a priori estimate of $\max|\alpha|$, we can choose $K$ small enough to determine $X\alpha$ uniquely, up to $O(h)$. Even if we do not, we can use the following argument. By \r{13a}, for each $\omega$,
\[
\frac{K^2}2 X\alpha(x,\omega) = -\Im \log( K^{-1} h^{1/2} \Lambda(u_\textrm{in})|_{x\cdot\omega=T}) +2\pi g(x, \omega) + O(h), 
\]
where $g$ is a possibly discontinuous function taking integer values only. Here, $\log$ is the logarithmic function of a complex argument, say in $\mathbf{C}\setminus (-\infty,0]$ with its principal branch there; which makes it a priori discontinuous. The error term can jump by $O(h)$ only; therefore, the discontinuities of the other two terms on the right must cancel each other.  Since $X\alpha=0$ for $|z|\gg0$, the first term on the right vanishes there up to $O(h)$; then $g=0$ there. This allows us to resolve the ambiguity caused by the multi-valued behavior of the log function, responsible for the term $2\pi g$ there. If we take any $z$ (with $\omega$ fixed) and connect it to some $z_0$ with $|z_0|\gg1$ with a continuous curve, then along that curve, the l.h.s.\ is continuous. Therefore, we can choose $g$ along that curve so that the r.h.s.\ stays continuous; i.e., we chose the appropriate branch of the logarithmic term by continuity. This does not depend on the curve because we work with data corresponding to some $\alpha$ by assumption, so the result is the l.h.s.

\begin{remark}\label{rem_g}
This construction works for non-linearities of the type $\alpha |u|^{m-1}u$, $m\ge2$ integer. Also, it works for $\alpha=\alpha(t,x)$; then one gets the light-ray transform of $\alpha$ instead of $X\alpha$. One can also have the Laplacian related to a Riemannian metric in \r{1}. The proof of the existence and well-posedness however, see Appendix~\ref{sec_gl},   when it holds, would take additional efforts. 
\end{remark}

\subsection{Lower order terms and justification} \label{sec_just}
To get to the lower order terms, write
\[
a=a_0+ha_1+h^2a_2+\dots, 
\]
then
\[
\begin{split}
|a|^2a &=  (a_0+ha_1)^2  (\bar a_0+h\bar a_1) + O(h^2) \\
&= |a_0|^2a_0+  2h a_0^2\bar a_1+ 2h |a_0|^2 a_1+ O(h^2).
\end{split}
\]
Then the next transport equation is, see \r{3} and \r{8},
\[
2 (1, \omega)\cdot (\p_t, \p_x)a_1 +\i   \left(\alpha a_0^2\bar a_1+ 2 \alpha|a_0|^2 a_1+ \Box a_0\right)=0
\]
with zero initial conditions when $x\cdot\omega\ll1$. 
In the notation of \r{9} and \r{11}, this takes the form
\[
2\frac{\d}{\d s}a_1+\i \alpha a_0^2\bar a_1+  2\i\alpha A^2 a_1+ \i \Box a_0=0, \quad a_1|_{s\ll0}=0.
\]
%I do not see an obvious way to get a nice formula for the solution. 
It is a linear non-homogeneous system for $\Re a_1$ and $\Im a_1$ with coefficients depending on $a_0$. 

The next transport equation takes the form
\[
2\frac{\d}{\d s}a_2+\i \alpha g_1(a_0,\bar a_0,a_1,\bar a_1) a_2+ \i \alpha g_2(a_0,\bar a_0,a_1,\bar a_1) \bar a_2+ \i \Box a_1=0,
\]
with zero initial condition again, where $g_{1,2}$ are  polynomials of third degree, etc.

After a finite number of steps, one gets an approximate solution $u_N$ solving
\be{PDEh}
\Box u_N + \alpha(x)|u_N|^2u_N=R_N = O(h^{N-1/2}),
\ee
with zero Cauchy data at $t=0$. The estimate on the right is in the uniform norm on $[0,T]\times B(0,R')$ with every fixed $R'>R$ and fixed $T>0$. If we use Sobolev norms, the r.h.s.\ is $O_{H^s}(h^{N-1/2-s})$.

\subsection{End of the proof of Theorem~\ref{thm_main}}
To complete the proof of Theorem~\ref{thm_main}, we will compare $u_N$ and the exact solution $u$ of \r{1}, \r{1c} guaranteed to exist by Theorem~\ref{thm_ex} and by Remark~\ref{remark_B1} after it. They solve \r{PDEh} and \r{1}, respectively, with the same initial condition \r{10a}. By the same remark, we can assume that the initial conditions are cut smoothly to a large ball of radius depending on $T$. Then the parametrix construction remains the same (and it is trivial for rays not intersecting $\supp\alpha$), so \r{PDEh} remainb true. 
We  apply Theorem~\ref{thm_stab2}. The $H^1$ norm of $R_N$ in \r{PDEh} is $O(h^{N-3/2})$, and to get an $O(h^{1/2})$ error, we need $N=2$ at least. By Remark~\ref{rem_final}, then  $\|u(t)-u_N (t)\|_{L^\infty}= O(h^{1/2})$. This remains true if we replace $u_N$ there with its principal term $u_1$ because its subsequent term contributes at most $O(h^{1/2})$ to $u$. This completes the proof of Theorem~\ref{thm_main} (notice the factor $h^{1/2}$ in \r{eq:main_thm}).

%\newpage
\section{Numerical experiments} 
\subsection{Take $\alpha$ to be a Gaussian} 
Our setup is the following. We take $n=2$ and solve \r{1} in the square $[-1,1]^2$ discretized into a $1,000\times 1,000$ mesh. For $\alpha$, we take an elliptic Gaussian centered at the origin with maximal value $1$ (i.e., not normalized for unit integral):
\be{al}
\alpha(x) = e^{-(x_1/0.2)^2 - (x_2/0.1)^2 }.
\ee
We take a plane wave starting at the bottom of the picture with corresponds to \r{10a} with $\chi$ another (1D) Gaussian
\be{13a'}
\chi(x) = Ke^{-(x_2/0.14)^2},
\ee
where $K>0$ controls the size of the initial data, see also \r{13a}. We refer to Figure~\ref{fig_setup} for an illustration. 
We take $K$ ``not too far'' from $K=1$ because if $K$ is too large or too small, we are
entering a different regime effectively. The regimes are defined as asymptotic ones; and we are taking only a few values of $h$. 
 We  take  $K=1$ first. If we take $K\gg1$, we are in the fully non-linear regime, if $K\ll1$, we are in the linear one.

\begin{figure}[ht] 
\begin{center}
	\includegraphics[trim = 0 100 0 50 , scale=0.18]{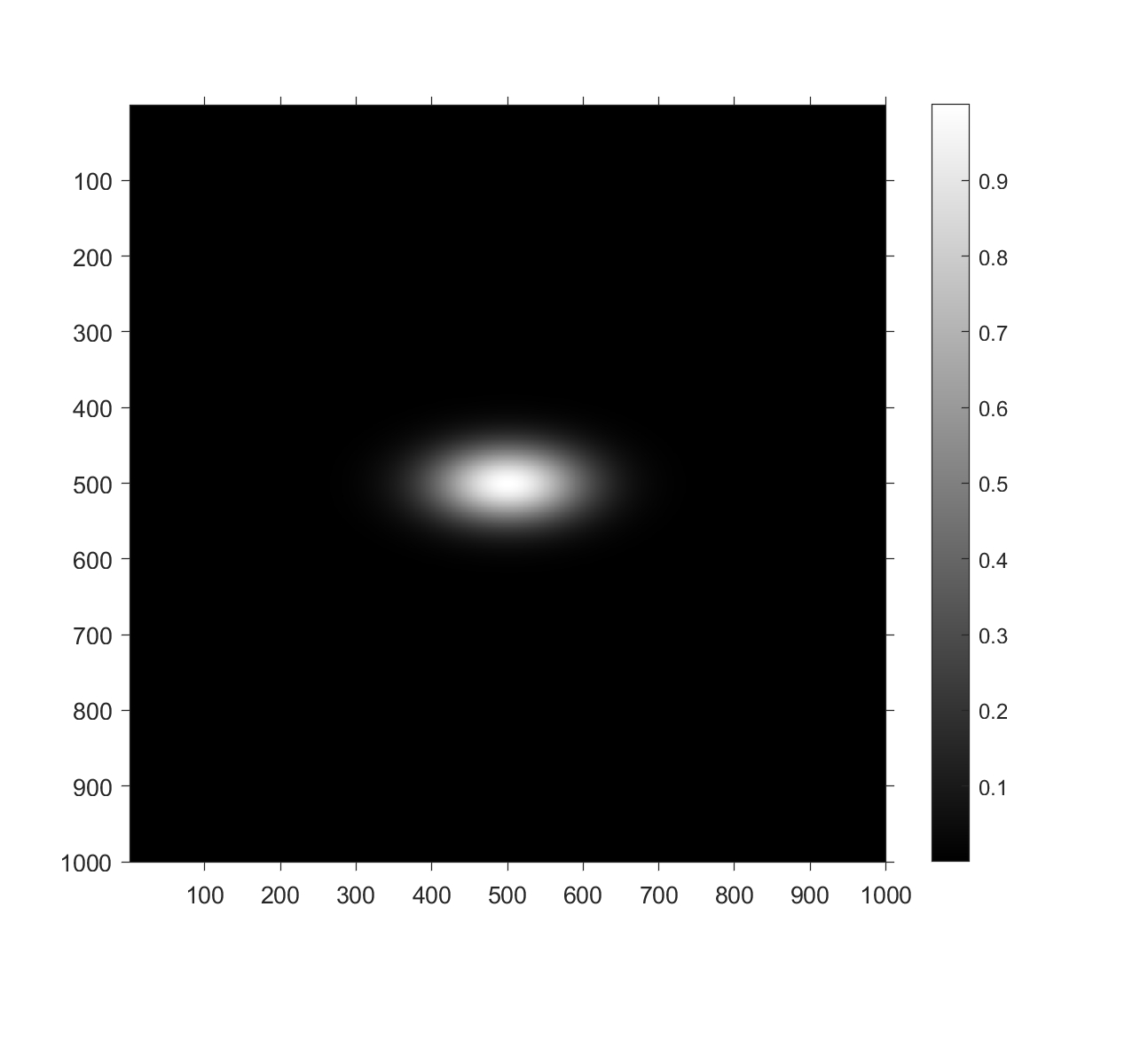}
	\includegraphics[trim = 0 100 0 50 , scale=0.18]{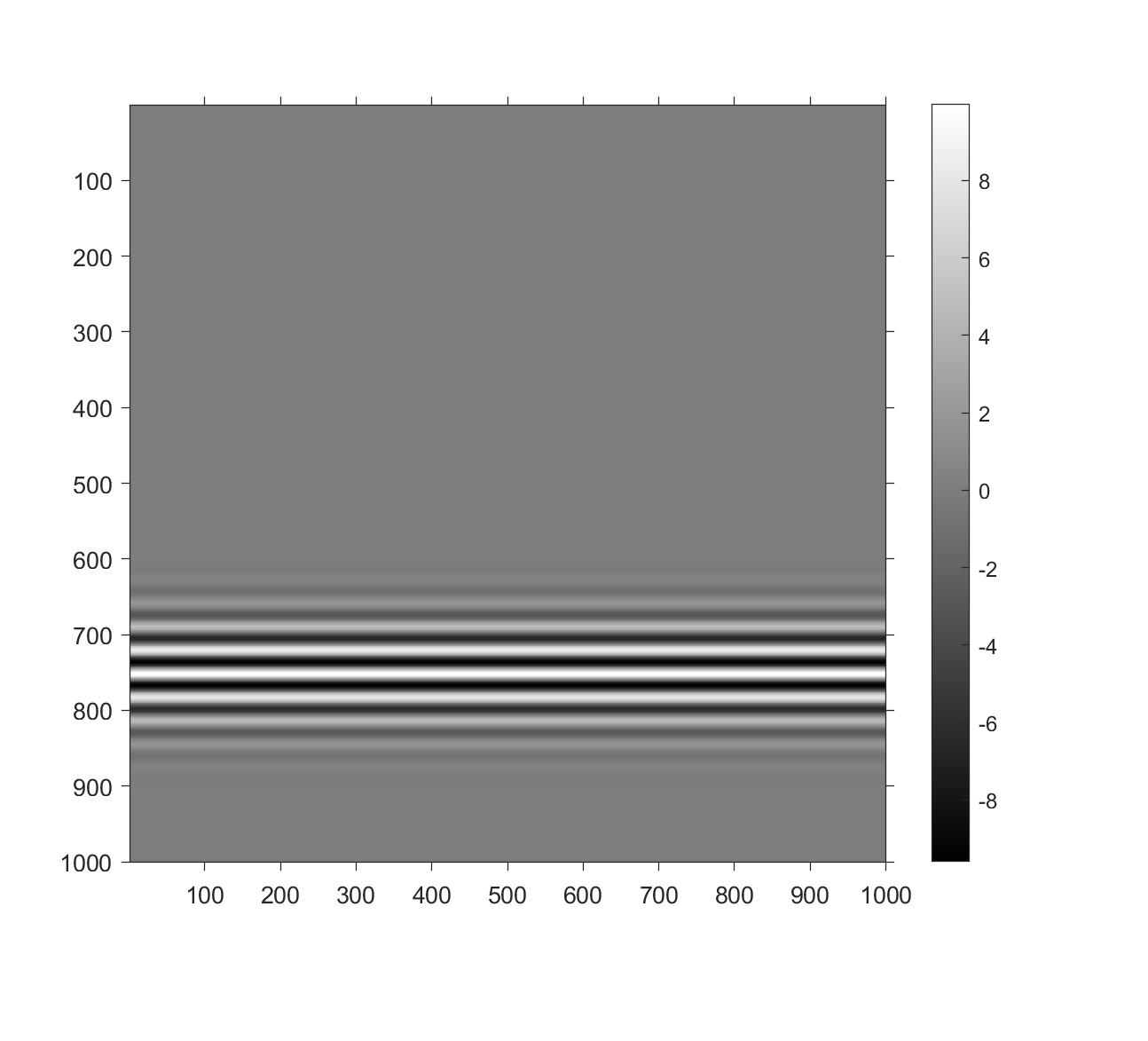}
\end{center}
\caption{\small   Left: the non-linearity $\alpha(x)$. Right: The initial condition $u|_{t=0}$ with $u_t|_{t=0}$ chosen so that it would travel upward (only).}
\label{fig_setup}
%\end{wrapfigure}
\end{figure}

The non-linearity $\alpha$ is ``essentially supported'' away from the ``support'' of $u_\text{\rm in}$ at $t=0$ (we shift $t$ so that $t=0$ corresponds to $t=-R=\delta$ in \r{1c}) and essentially leaves the ``support'' of $\alpha$ at time $t=1$, see Figure~\ref{fig_K1}. Then $\Lambda$ in this case maps the solution at $t=0$ to that at $t=1$, the latter localized where $\chi(-1+x\cdot\omega)$ is ``essentially supported'', which is also true for the solution $u$ as well.  

We take $h=0.01$ first. This $h$ is not small enough relative to the standard deviation of $\chi$ and the pattern looks a bit like a coherent state in vertical direction. We really want the wavelength to be much smaller than the ``support'' of $\chi$. Here it is borderline so. 
We do not show the linear case ($K\ll1$). The pattern just moves upward, nothing interesting. 

We take $K=1$ first, see Figure~\ref{fig_K1}.  
\begin{figure}[ht]
\begin{center}
	\includegraphics[trim = 0 100 0 50 , scale=0.15]{IC1} 
	\includegraphics[ trim = 0 100 0 50 , scale=0.15]{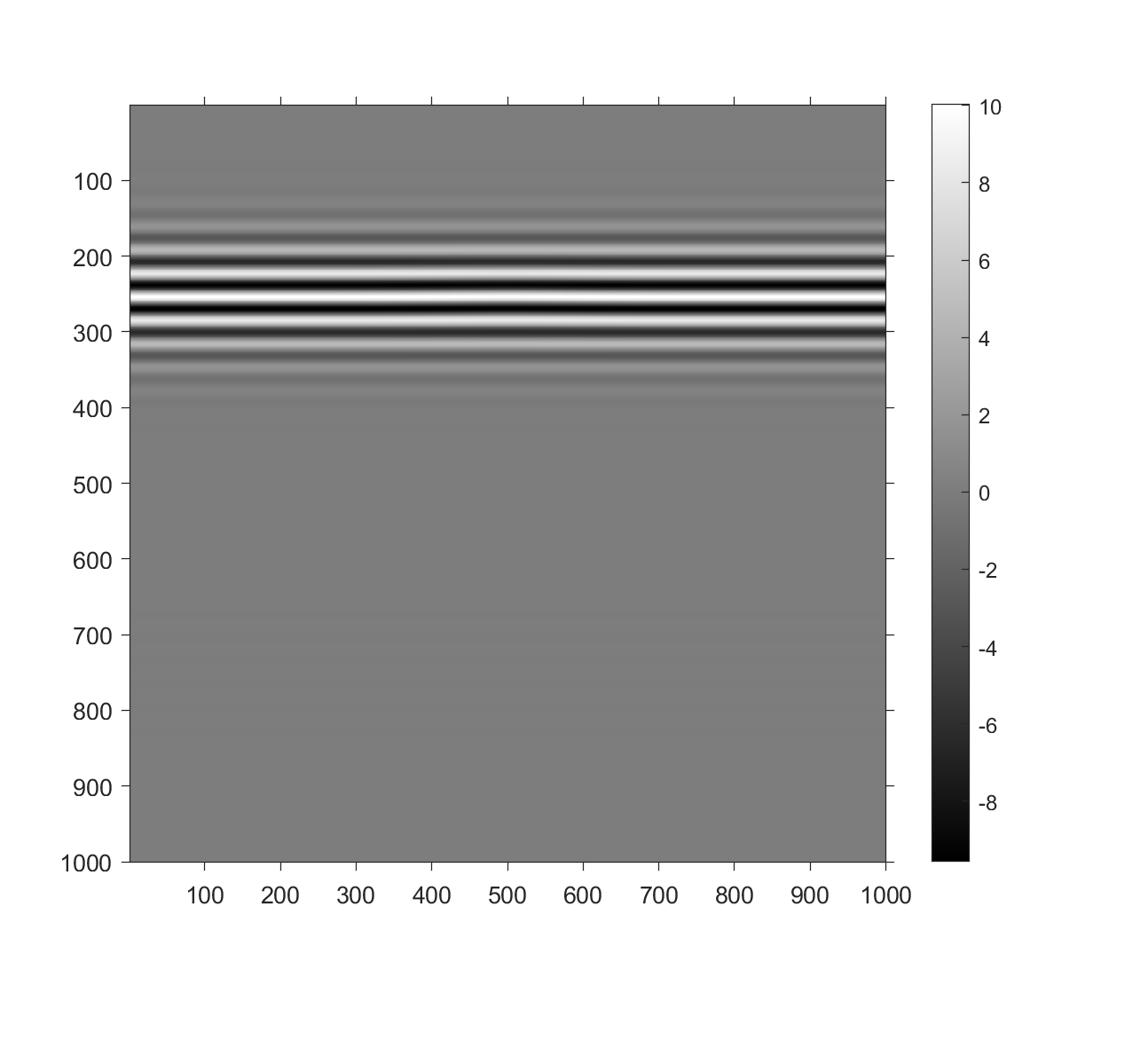}
	\includegraphics[trim = 0 50 0 50  , scale=0.137]{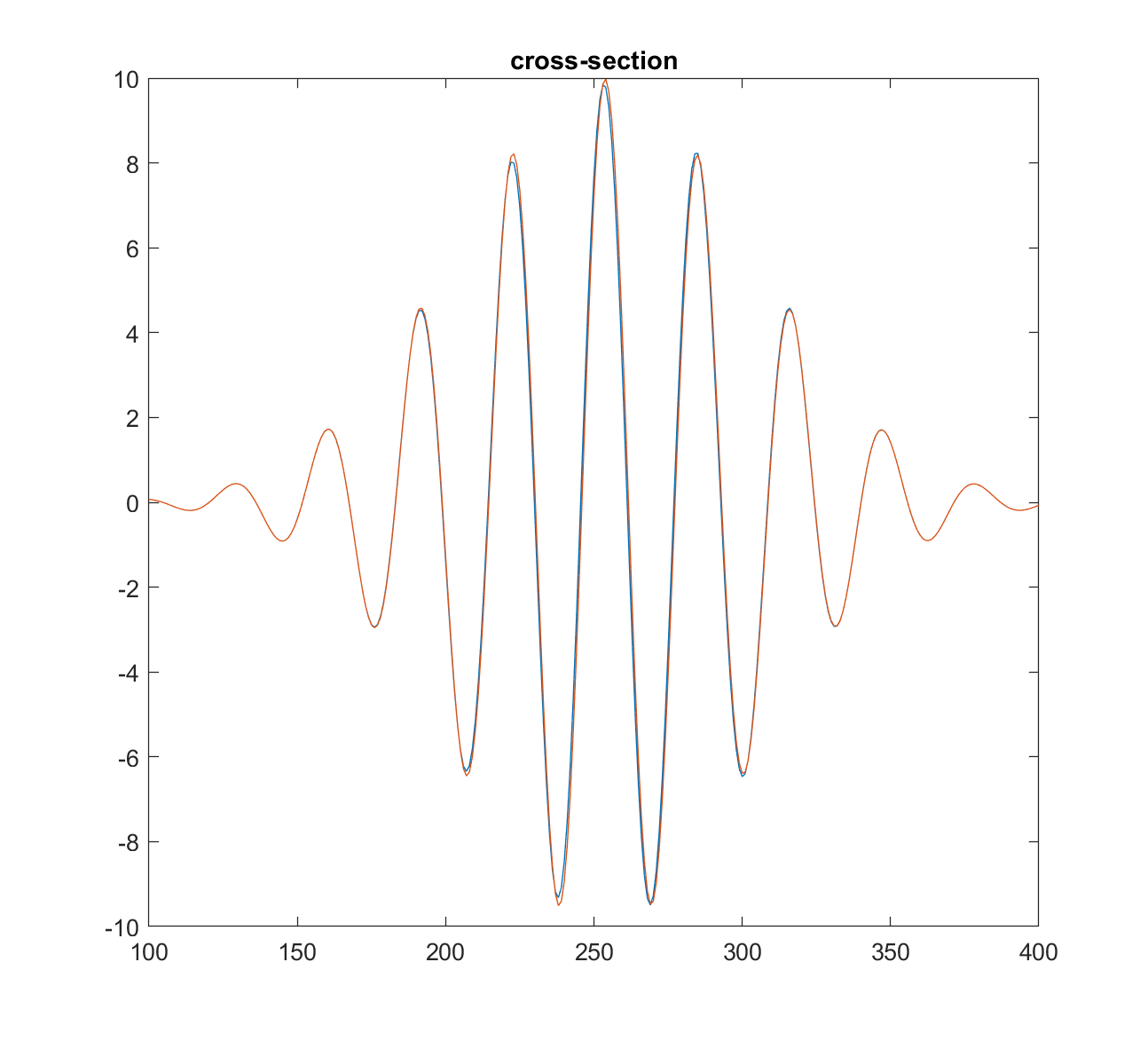}
\end{center}
\caption{\small $K=1$.  Left: Initial condition. Center:   $\Re u(1,x)$ at time $t=1$ . Right: Plot of a vertical cross section of $\Re u(1,x)$ through the center and that of the linear solution (the red curve).  The left end corresponds to the top of the center plot. They are very close. }
\label{fig_K1}
\end{figure}
The solution at $t=1$ looks very close to that of the linear PDE. A plot of the cross-section, with that of the linear solution super-imposed reveals a very slight phase shift. In Figure~\ref{fig_K1a} we show a crop of that cross-section and then 
\be{14}
\Im \log h^{1/2}e^{\i(t-x\cdot\omega)/h}u|_{t=1} \approx \frac12K^2\chi^2(-1+x_2)X\alpha(x_1),\quad K=1,
\ee
where $X\alpha(x_1)$ is the X-ray transform of $f$ along vertical lines. The picture confirms our calculations, for example the maximal value of the blob there is very close to what the right-hand of the formula above predicts. 

\begin{figure}[ht]
\begin{center}
	\includegraphics[trim = 0 100 0 50 ,scale=0.137]{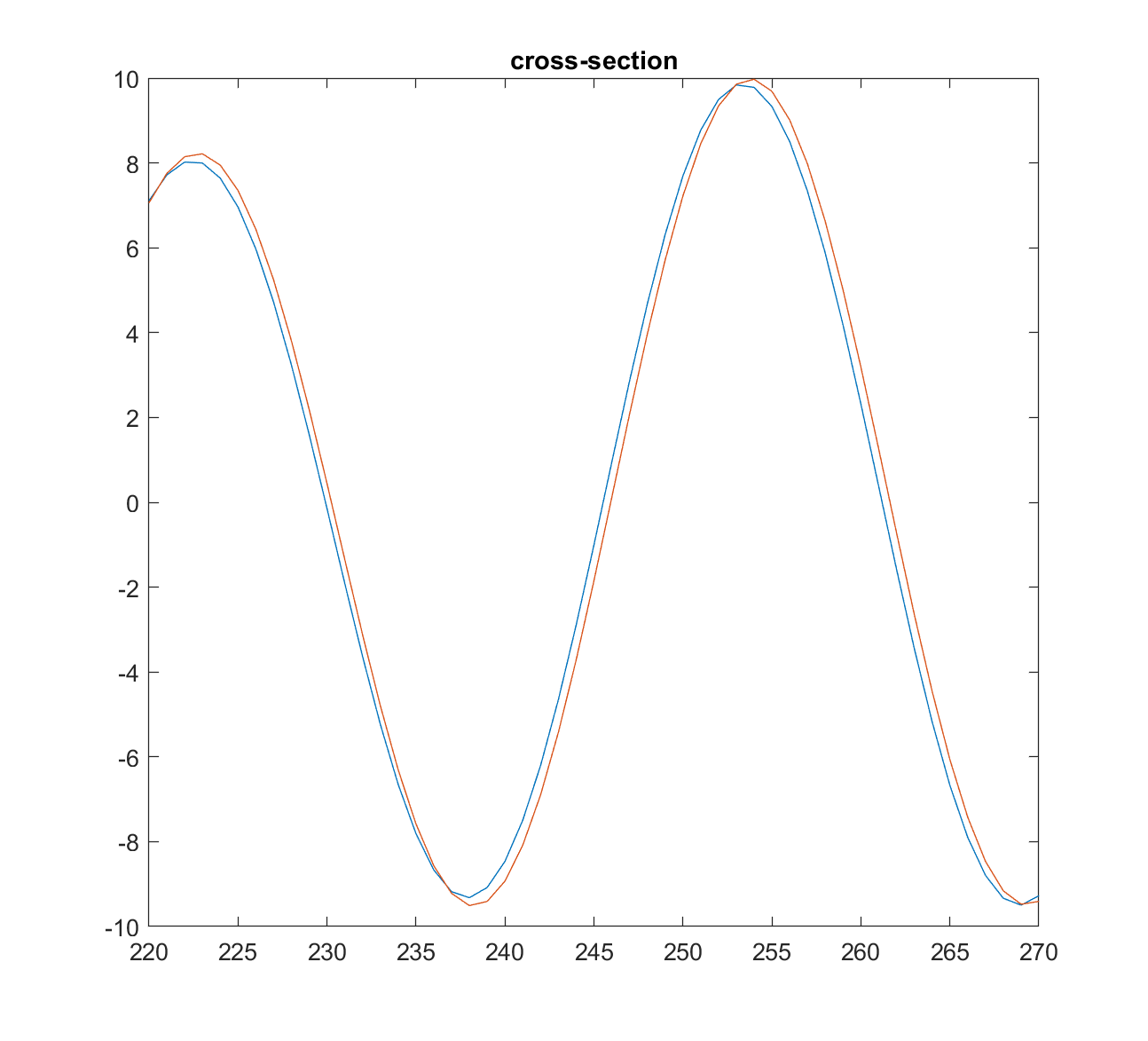}
	\includegraphics[trim = 0 150 0 50 , scale=0.15]{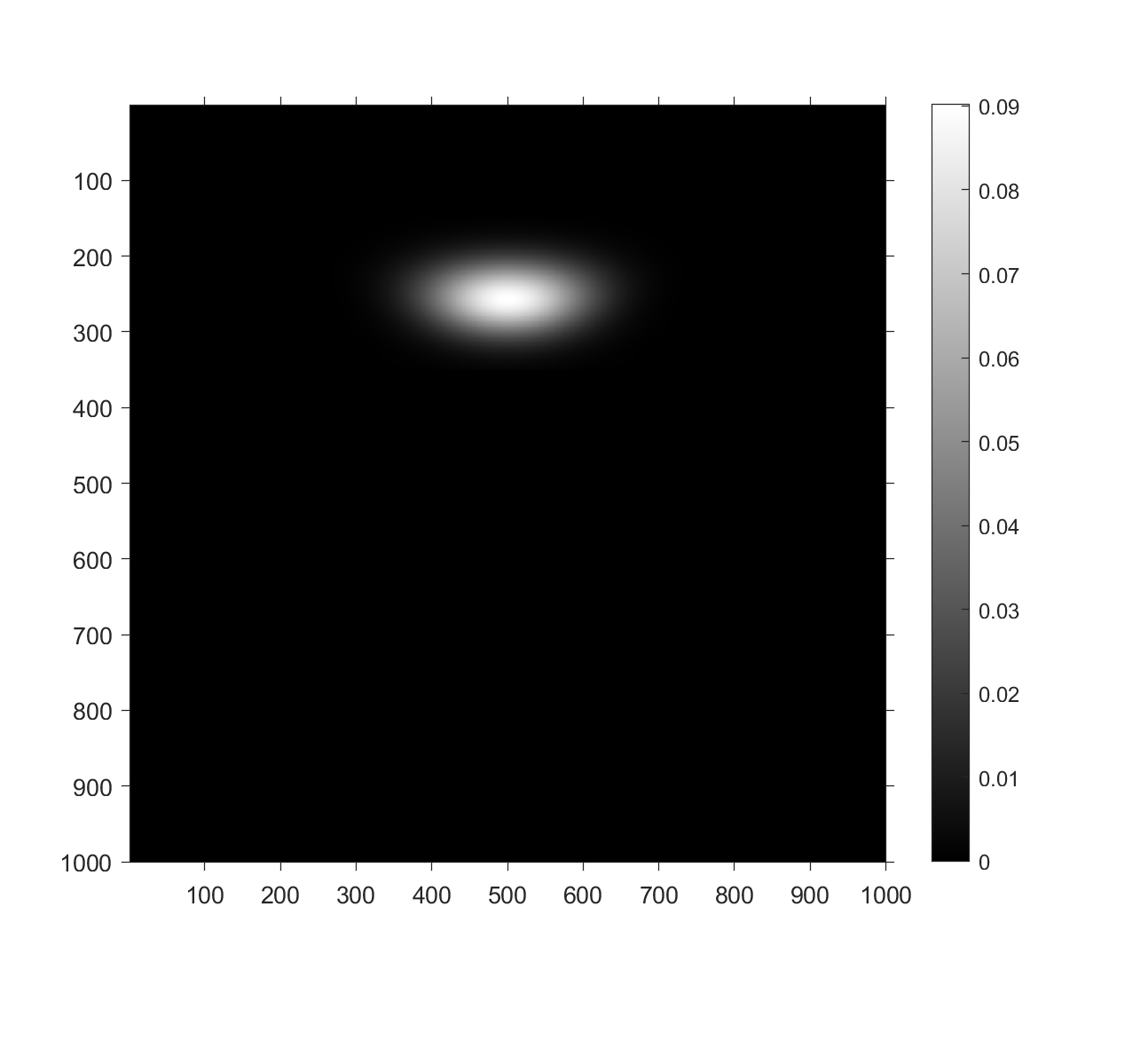}
\end{center}
\caption{\small   Left: the cross-section above, zoomed in. Right: the phase shift.}
\label{fig_K1a}
%\end{wrapfigure}
\end{figure}

The maximum in \r{14} is about $0.088$. That gives us distances between the zeros of the osculations of the order of $(2\pi\pm 0.088)h$, i.e., maximal relative shift of the order of $0.88/(2\pi)\approx 0.014 = 1.4\% $, see \r{16}. This is what we see in Figure~\ref{fig_K1a}. 

In the next example in Figure~\ref{fig_K5}, $K=5$. This is getting a bit close to the fully non-linear case. The phase shift is much more apparent now. We are multiplying by $K^2=25$ in \r{14} to get a relative shift $\approx 35\%$. The amplitude of the non-linear oscillations drops visibly by about 20\%, which is an effect that can be explained either by the next term in the expansion or by the fact that we are getting closer to the fully non-linear regime. 
\begin{figure}[ht]
\begin{center}
	\includegraphics[trim = 0 130 0 50 , scale=0.15]{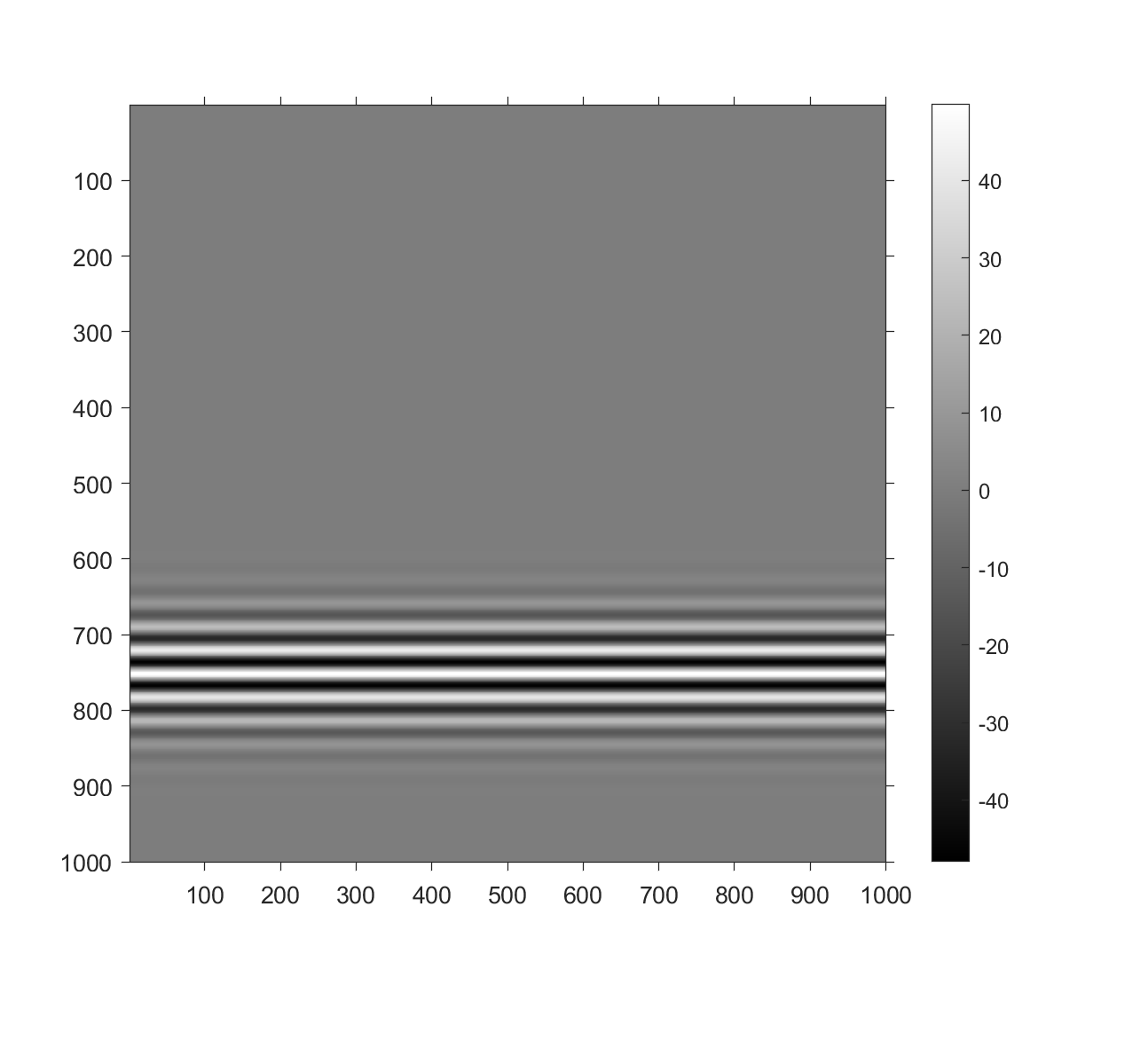}
	\includegraphics[ trim = 0 130 0 20 , scale=0.15]{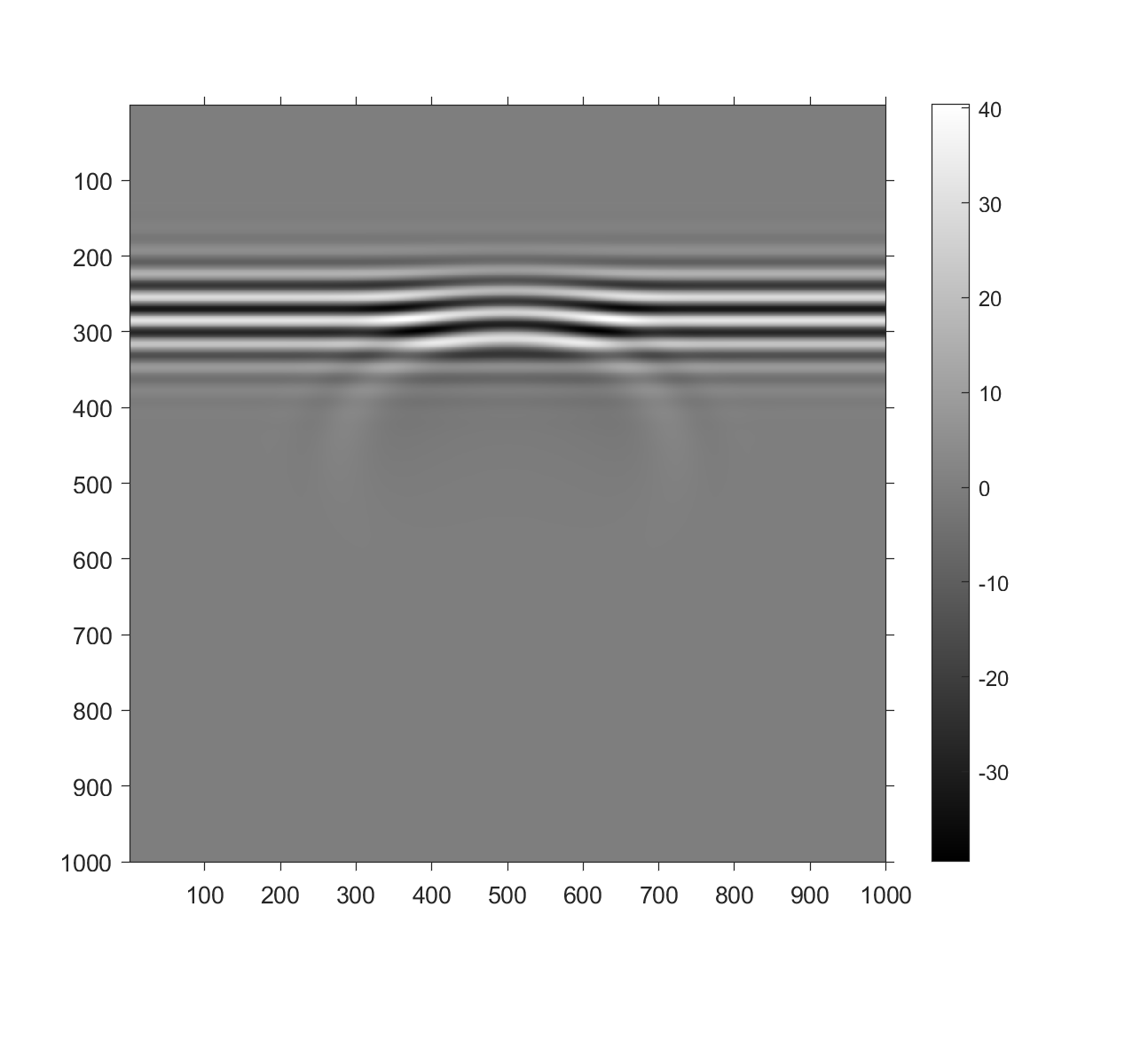}
	\includegraphics[trim = 0 90 0 20, scale=0.137]{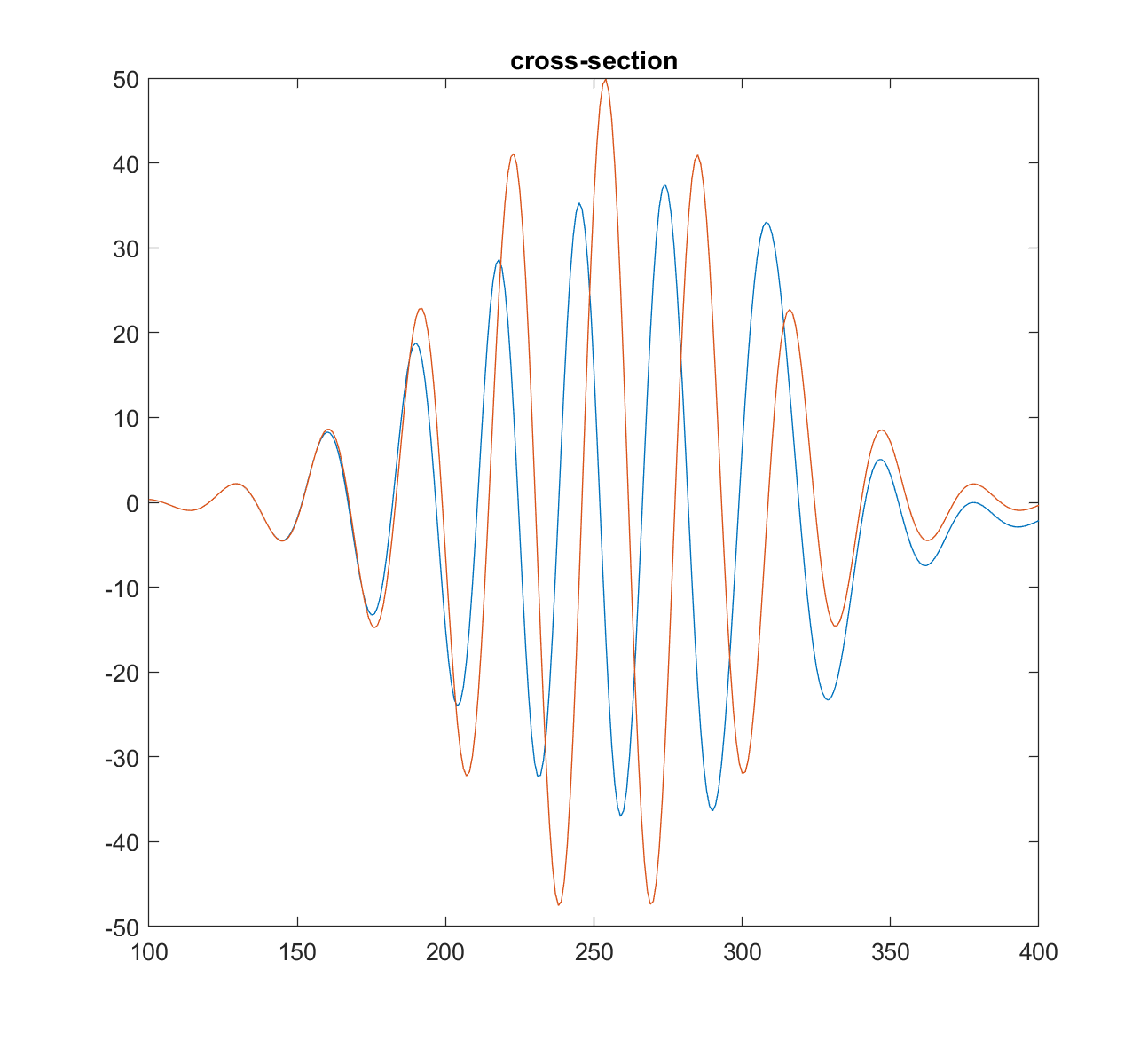}
\end{center}
\caption{\small $K=5$, $h=0.01$.  Left: Initial condition. Center:   $\Re u(1,x)$ at time $t=1$ . Right: Plot of a vertical cross section of $\Re u(1,x)$ through the center and that of the linear solution (the red curve). 
}
\label{fig_K5}
\end{figure}

We decrease $h$ now by half: $h=0.005$. Then the reflection above disappears, and the relative shift is similar, see Figure~\ref{fig_K5_h005}. The phase shift computed by \r{14} looks pretty much like in Figure~\ref{fig_K1a}, left, but even more symmetric, with maximal value $2.1-2.2$, which is what that formula predicts. We are mainly in the weakly non-linear regime now. If we take $h$ even smaller, we are really there, of course, regardless of the value of $K$ which stays fixed.

As a test, one can plot $|u(t,x)|$ and compare that to the linear solution which is just $u_\text{\rm in}|_{t=1}$. If the approximation \r{12} is a good one, one should not see the non-linear effect (up to an $O(h)$ error) because the latter is in the phase only. In this case, it is almost not-visible (plot not shown), i.e., the absolute values of the linear and the non-linear solutions are very close, which is a good confirmation. 

\begin{figure}[ht]
\begin{center}
	\includegraphics[trim = 0 150 0 0 , scale=0.11]{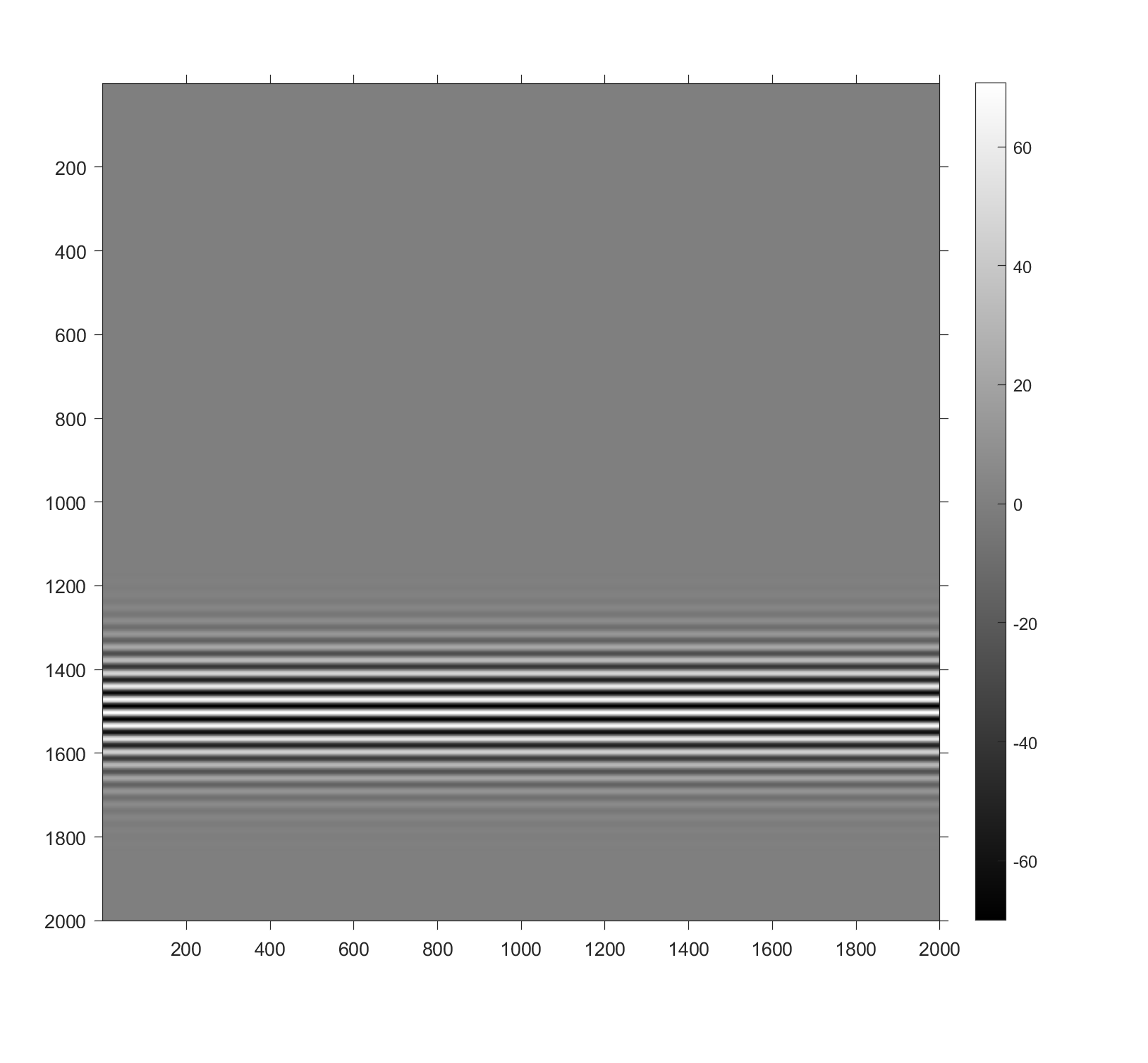}
	\includegraphics[ trim = 0 150 0 0 , scale=0.11]{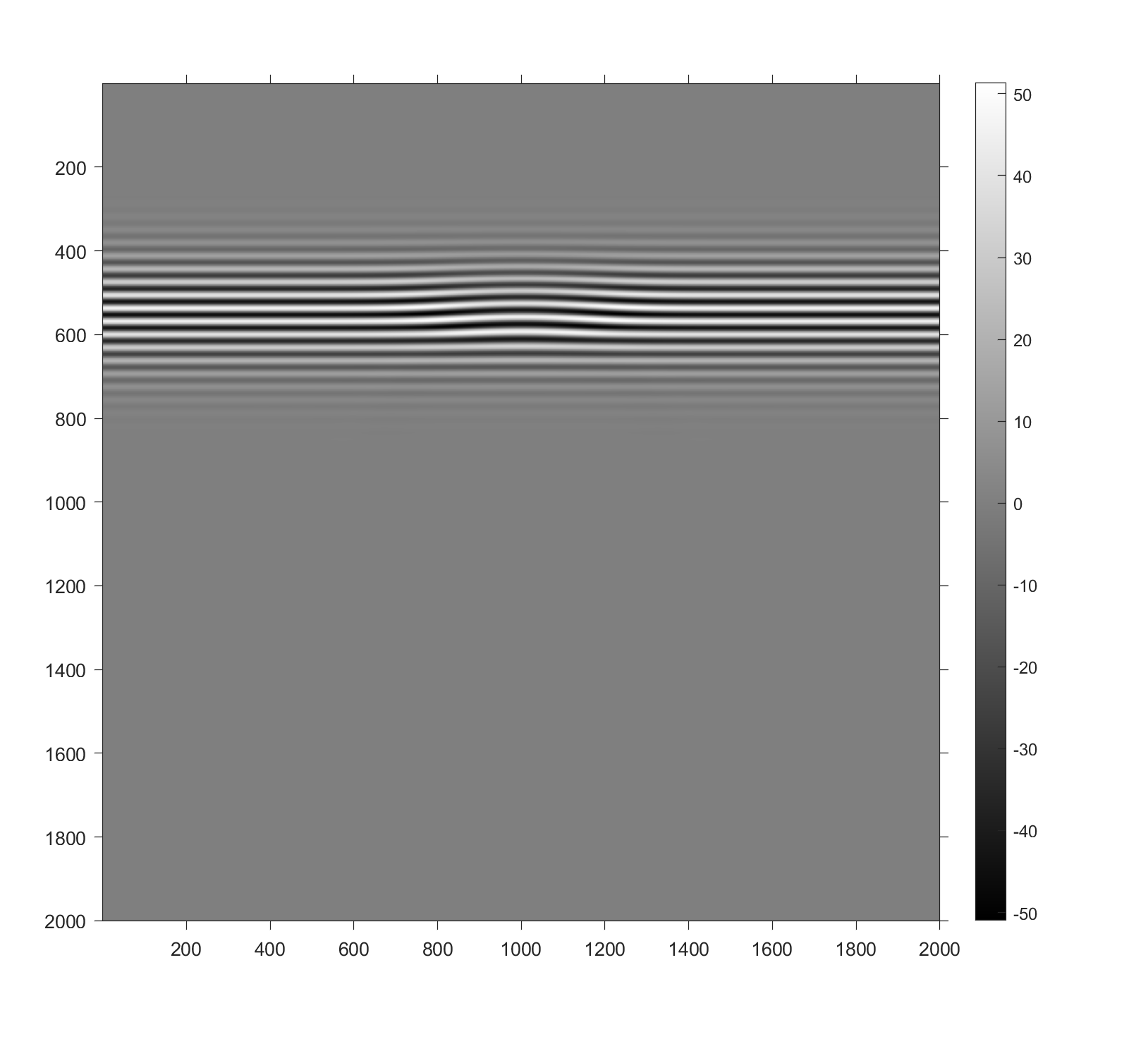}
	\includegraphics[trim = 0 130 0 0, scale=0.105 ]{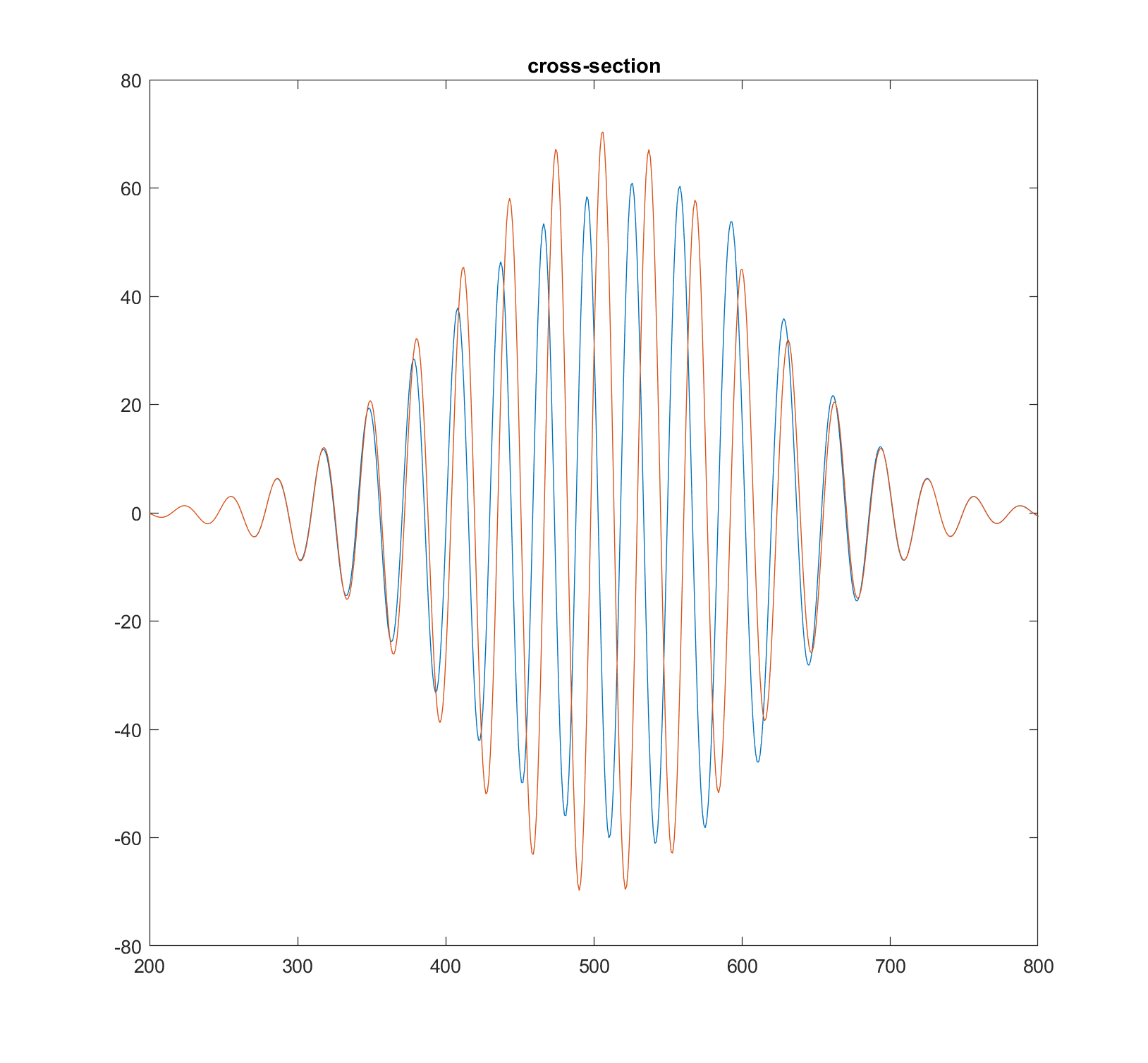}
\end{center}
\caption{\small $K=5$.  As above but $h=0.005$ now. %They are very close to each other.
}
\label{fig_K5_h005}
\end{figure}

For numerical purposes it is good to integrate the phase shift \r{14} w.r.t.\ $x_2$ over $\supp\chi$ instead of fixing $x\cdot\omega=R$ (in this case, taking $x_2=1$) as in \r{13a}. This corresponds to integral in vertical direction over the top $1/3$ or so in Figure~\ref{fig_K5_h005}, for example. Then we get a multiple of $X\alpha$:
\be{data}
\text{Data}(x_1)= C X\alpha, \quad C:=\frac12 K^2\int\chi^2(s)\,\d s,
\ee
where ``Data'' is \r{14} integrated in $x_2$. 
 In Figure~\ref{fig_K5_h005_p}, we show a plot of the theoretical data vs.\ the computed one, with $K=5$ and $h=0.005$. There is a very close match. 

\begin{figure}[ht]
\begin{center}
	\includegraphics[trim = 0 50 0 0 , scale=0.25]{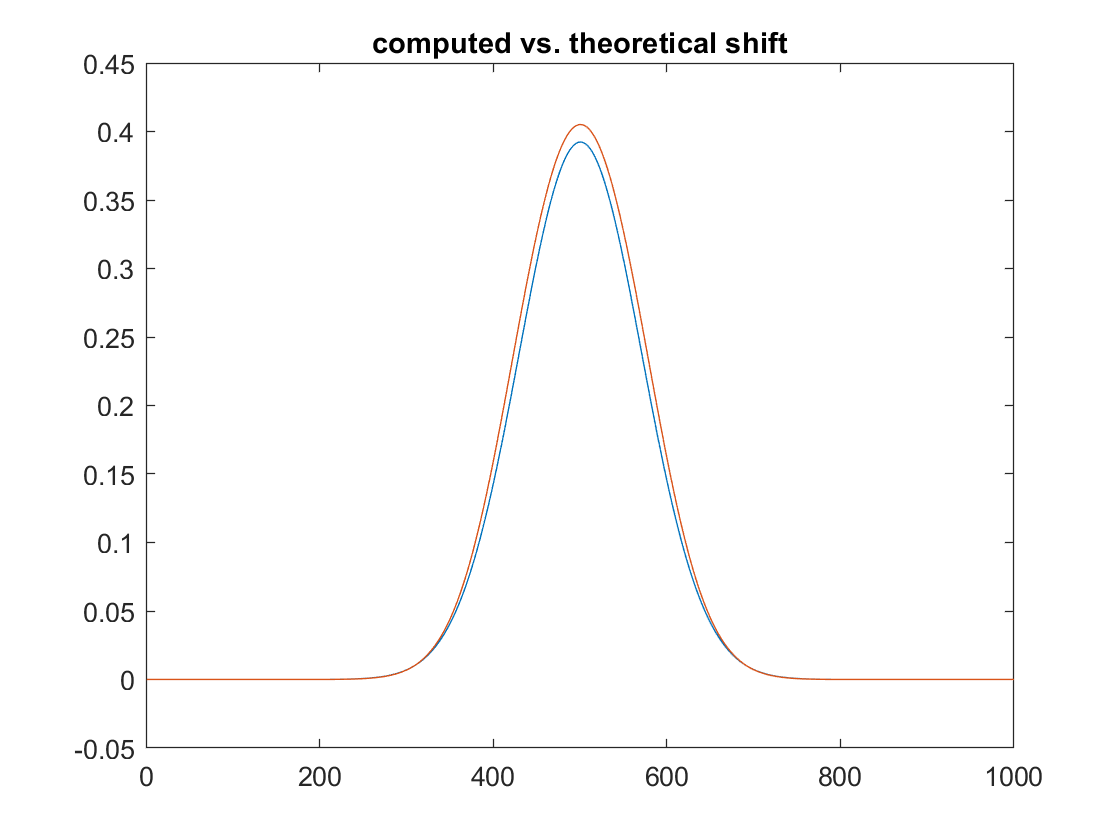}
\end{center}
\caption{\small $K=5$, $h=0.005$.  The theoretical data (in blue) vs.\ the computed one (in red).  
}
\label{fig_K5_h005_p}
\end{figure}

It follows from \r{12} that the absolute phase shift is 
\[
-\frac{1}2 hK^2 \chi^2(-t+x\cdot\omega) X\alpha, \quad \text{where $t=1$}.
\]
Recall that in \r{2} we have the prefactor $Kh^{-1/2}$ because of \r{13a'}. For $h\ll1$, the absolute phase shift tends to zero. 
If $\alpha=0$, the fundamental quasi-period is $2\pi h$. 
Therefore, the relative phase shift is
\be{16}
-\frac1{4\pi}  K^2\chi^2(-1+x\cdot\omega )X\alpha,
\ee
which is $h$-independent. 
The negative sign says that the wave ``speeds up''  which we see numerically as well. Its maximum is where $ -1+x\cdot\omega=1 $ (for our choice of $\chi$), i.e., at the center of the outgoing wave, and this can be observed numerically as well, see Figure~\ref{fig_K5_h005}, right.

\subsection{Take $\alpha$ to be the Shepp-Logan phantom} 
We take $\alpha$ to have be the Shepp-Logan phantom, blurred a bit. 
It occupies the $[-0.5,0.5]^2$ square in $[-1,1]^2$, see Figure~\ref{fig_SL1}. On the right, we plot the X-ray transform along vertical lines as computed by the phase shift (in red) vs.\ the theoretical one in blue. One can see the Gibb's phenomenon (or aliasing) since the wavelength is not small enough to capture the boundary. This is quite clear by just plotting the waves and the phantom one over the other. The analysis of the needed wavelength,  related to the frequency content of the phantom, would be complicated since for very sharp boundaries, one can see diffraction affects. In any case, if the phantom is smooth (which it is), as $h\to0$, we would get an accurate approximation. 
\begin{figure}[ht]
\begin{center}
	\includegraphics[trim = 0 50 0 0 , width=0.3\textwidth]{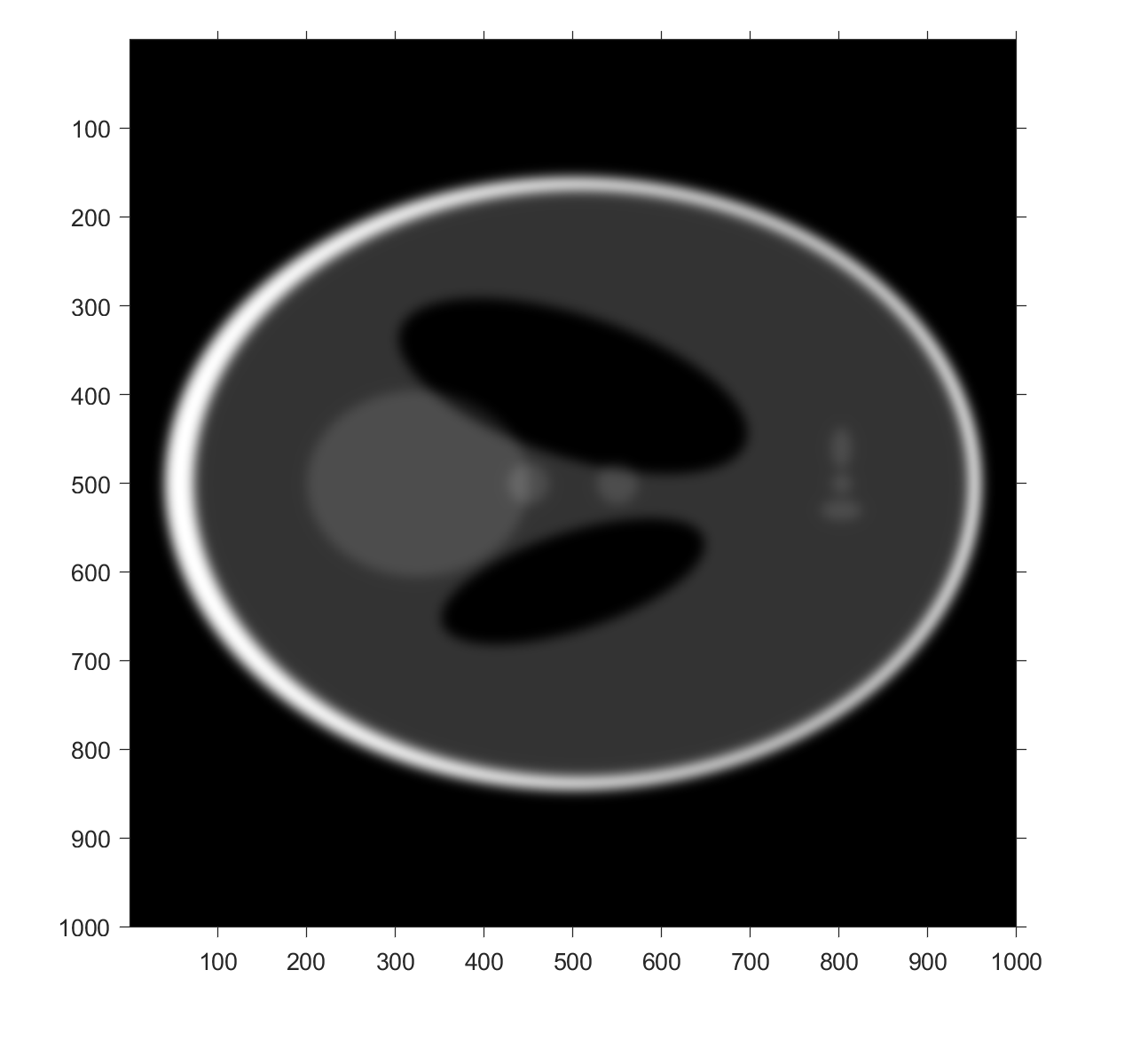}
	\includegraphics[ trim = 0 50 0 0 , width=0.38\textwidth]{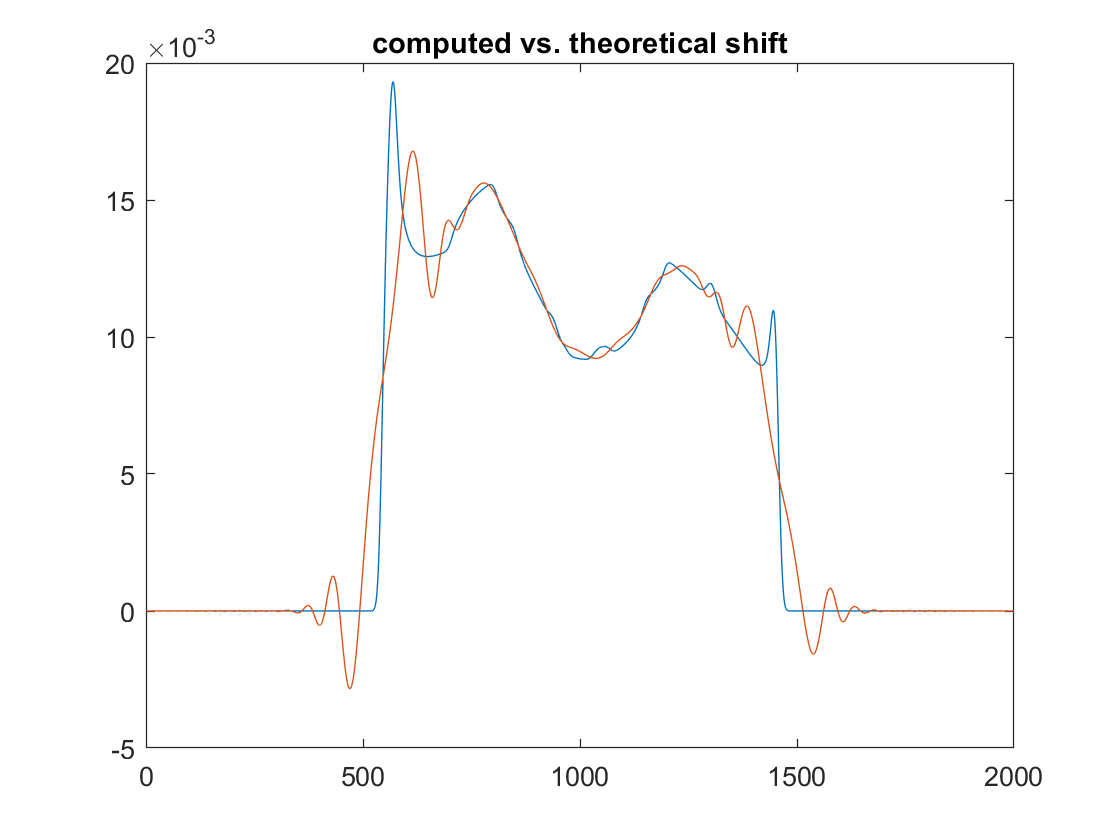}
\end{center}
\caption{\small $K=1$, $h=0.0025$. Left: the Shepp-Logan phantom. Right: The theoretical data in blue (its X-ray transform along vertical lines) vs.\ the computed one in red.  %They are very close to each other.
}
\label{fig_SL1}
\end{figure} 

\subsection{Stability and noise considerations} 

Theorem~\ref{thm_stab} illustrates that the direct problem is stable under small perturbations of the initial conditions. 
In the next example, we add $10\%$ noise (relative to the maximal value) to the initial condition \r{10a}. The waves propagate without blowing up and look weakly affected by the noise (plot not shown).   
In Figure~\ref{fig_K5_h0025_noise2}, we plot the theoretical shift in blue computed by \r{data} vs.\ the computed one in red.

\begin{figure}[ht]
\begin{center}
	\includegraphics[trim = 0 0 0 0, scale=0.2 ]{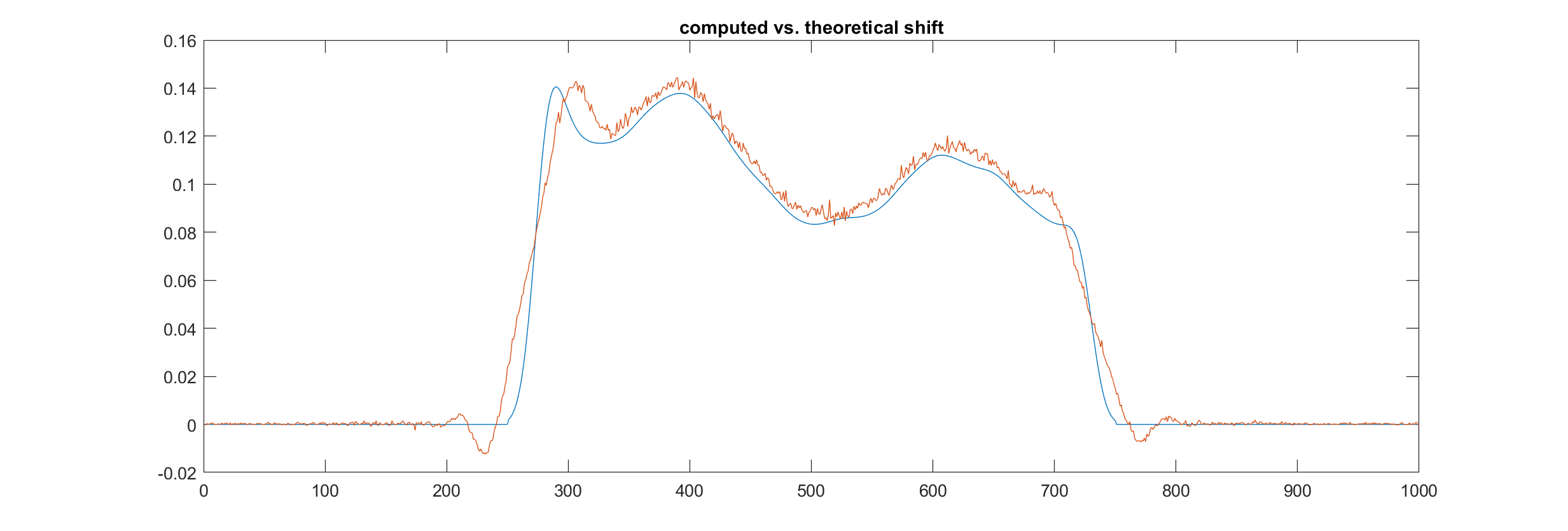}
\end{center}
\caption{\small $K=5$, $h=0.0025$, $10\%$ noise added to the initial condition.  The theoretical data (in blue) vs.\ the computed one (in red).   
}
\label{fig_K5_h0025_noise2}
\end{figure}

\subsection{Recovery}
We choose a $2,000\times 2,000$ grid discretizing $[-1,1]^2$, with $h=0.0025$, $K=1$ and the phantom shown below. We take the data by rotating the image by 1 degree from $0$ to $179$ degrees. Then we extract $X\alpha$ and invert $X$. The result is shown in Figure~\ref{fig_recovery}. The non-linearity is included in $[-0.4,0.4]^2$, and this is what we plot. One can see that the resolution is relatively low, and the reconstructed image is blurred more than the original. The wavelength is approximately $2\pi h\approx0.016$, so the displayed squares are approximately $51\times 51$ in wavelength units. Full analysis of the resolution of this method however is well beyond the scope of this paper. If we choose $\alpha$ with sharper edges (but still not with jumps), we observe diffraction effects of the waves from them. When $h\ll1$ those effects would disappear but the computational cost then becomes too high. 

 \begin{figure}[ht]
\begin{center}
	\includegraphics[trim = 0 50 0 0, scale=0.2 ]{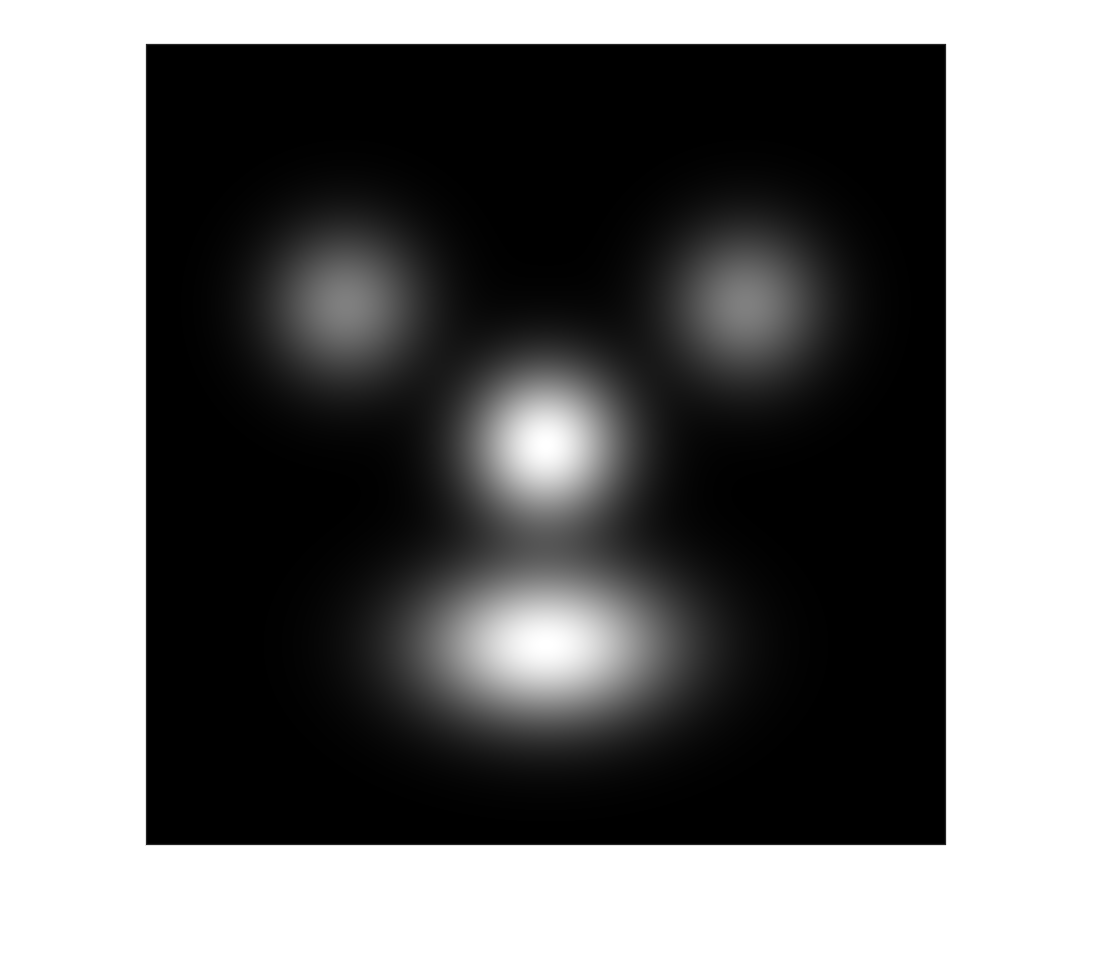}
\includegraphics[trim = 0 50 0 0 , scale=0.2]{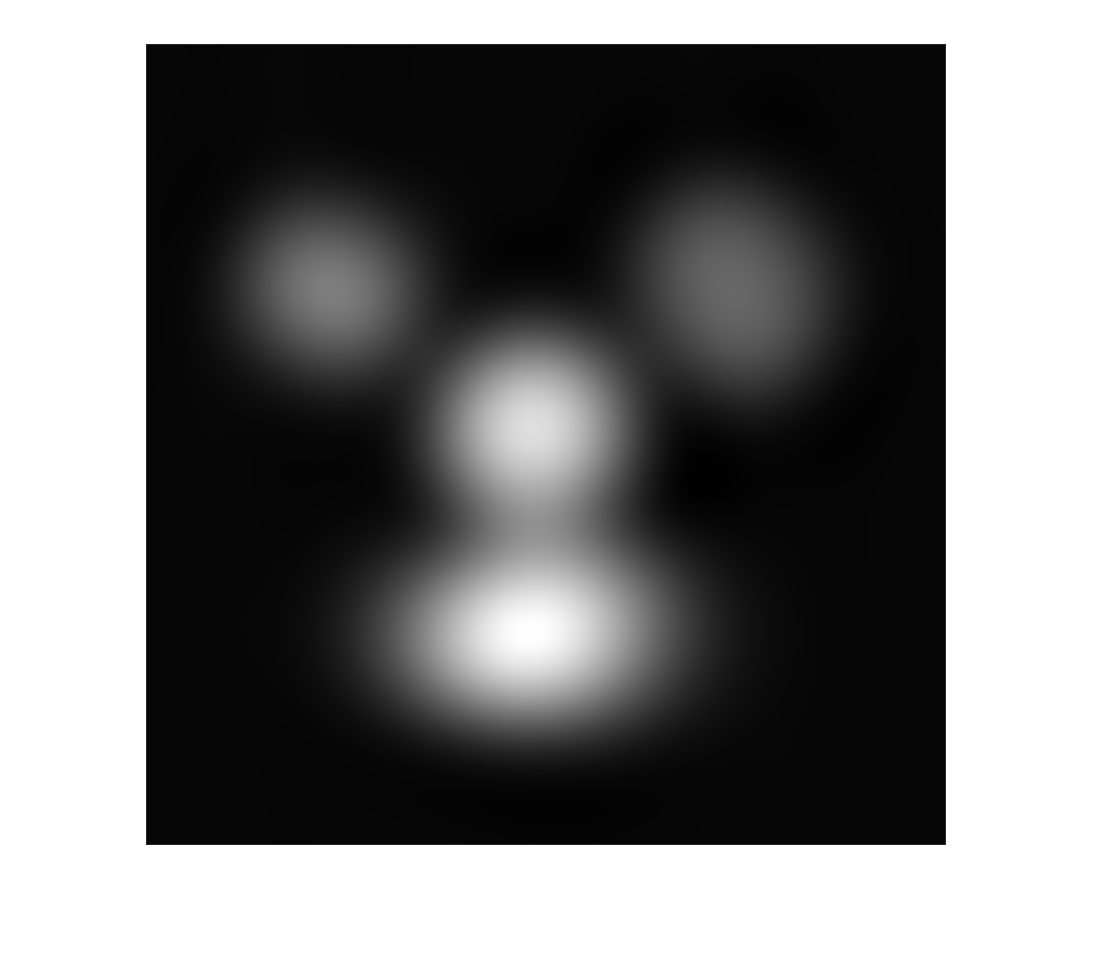}
\end{center}
\caption{\small $K=1$, $h=0.0025$, original on the left, recovery on the right.   
}
\label{fig_recovery}
\end{figure}

%\newpage
\section{Probing with real valued waves}\label{sec_real}
\subsection{Geometric optics} 
When the Cauchy data $\mathbf{u}(0)$ is real valued, the solution is real and the PDE can be written as 
\[
u_{tt}-\Delta u+ \alpha(x) u^3=0,\quad (t,x)\in \R_t\times\R_x^n.
\]
In some works, this is the form of the semilinear wave equation with cubic non-linearities. One may think of real valued solutions being the physical ones in certain applications. We take incoming waves of the kind
\be{10a'}
u_\textrm{in}= \cos\frac{-t+x\cdot\omega}{h} h^{-1/2}\chi(-t+x\cdot\omega), \quad \omega\in S^{n-1}.
\ee
Since there is no principle of superposition, we cannot take $\cos$, then $\sin$ functions above and generate the solution with the complex incoming wave \r{10a}. It is known that in this case we get non-trivial harmonics corresponding to values of $k$ below different from $|k|=1$,  see for example \cite{Metivier-Notes}. The cubic type of non-linearity we have generates odd harmonics only. In Figure~\ref{fig_spectrum} we plot the power spectrum of a vertical slice of $u(1,x)$ with a non-linearity $\alpha$ as in \r{al}; one can clearly see the peaks created by the harmonics with $|k|=1,3,5$. The higher order ones cannot be realized with the chosen level of discretization.  

\begin{figure}[ht]
\begin{center}
	\includegraphics[trim = 0 50 0 50 , scale=0.28]{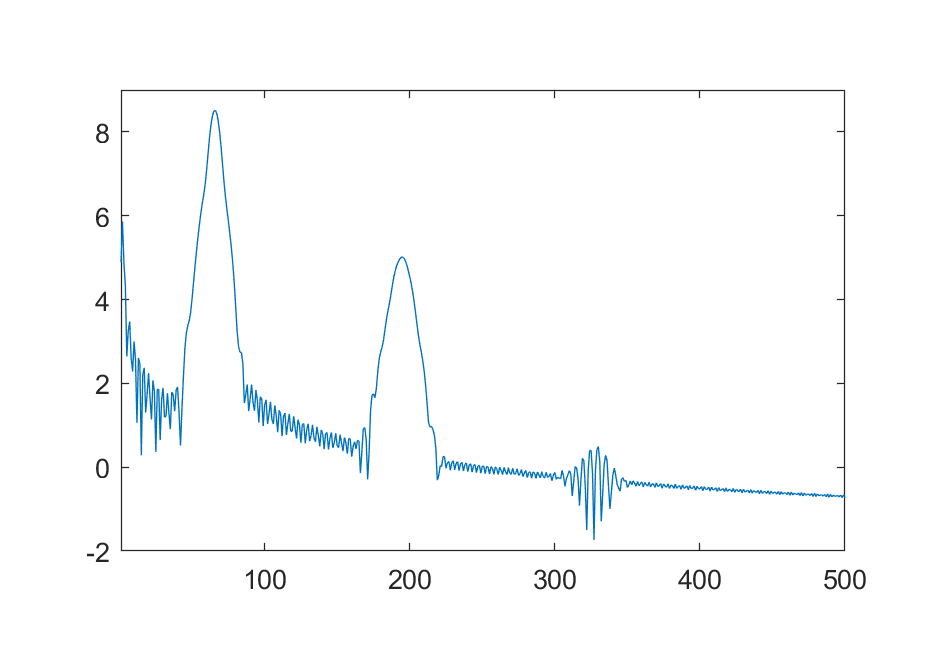}
\end{center}
\caption{\small  $\log |\hat u(t=1,x_1=0,\,\cdot\,)|$, i.e., $u|_{t=0}$ is restricted first to the vertical line $x_1=0$ through the center, with $\alpha$ a Gaussian as in \r{al}. One can see a peak at the main frequency $k\approx 66.7$ corresponding to the incoming one and the harmonics $3k$ and $5k$.
}
\label{fig_spectrum}
\end{figure}

The weakly non-linear ansatz in this case is of the type $u=U(\phi/h,t, x,\omega, h)$ with $U$ periodic in its first variable of period, say, $2\pi$, see, e.g., \cite{Metivier-Notes}. Expanding the oscillatory factor in Fourier series, we assume 
\be{R1}
u\sim  \sum_{k\in \mathbf{Z}\setminus 0}  e^{\i k \phi/h}h^{-1/2} a^{(k)}(t,x,\omega,h),
\ee
compare with \r{2}, where each $a^{(k)}$ has an expansion $a^{(k)} = a^{(k)}_0+h a^{(k)}_1+\dots$. Using \r{3}, we see that we can take the phase $\phi$ satisfying the same eikonal equation \r{5}; and we take the phase $\phi=-t+x\cdot\omega$ because this is what is dictated by the initial condition \r{10a'}. Plug \r{R1} into the equivalent of \r{3} (no absolute value in $|u|^2$ there) to get 
\be{exp}
\begin{split}
-\sum_k &e^{\i k\phi/h} 2\i k\frac{\d}{\d s}\left(  a^{(k)}_0+h a^{(k)}_1+\dots\right)  +\alpha  \Big( \sum_k  e^{\i k\phi/h} \Big( a^{(k)}_0+h a^{(k)}_1+\dots\Big)\Big)^3 \\
&+ h\sum_k e^{\i k\phi/h}  \Box \left(  a^{(k)}_0+h a^{(k)}_1+\dots\right)=0.
\end{split}
\ee

To simplify the notation, we will set  $\a_{k}=a_0^{(k)}$. 
The first order transport equation derived from \r{exp}, in the characteristic coordinates $(s,y)$, see \r{10} and \r{11},  takes the form
\be{9'}
2k\frac{\d}{\d s}\a_{k}+\i\alpha \!\!\! \!\!\! \!\! \sum_{k_1+k_2+k_3=k} \a_{k_1} \a_{k_2} \a_{k_3}  =0, \quad k=0, \pm1, \dots,
\ee
with initial conditions
\be{R2}
\a_{1}=\a_{-1}=\frac{A}2, \quad \a_{k}=0, \quad k\not=-1,1, \quad \text{for  $s\ll0$}. 
\ee
Multiply \r{9'} by $\a_{-k}$ and then do the same with $k$ and $-k$ swapped, subtract and sum up to get
\[
\begin{split}
\sum_{k=1}^\infty 2 \frac{\d}{\d s}\left( k \a_{k} \a_{-k}\right) &= \ \ \ 
\i\alpha \!\!\! \!\!\! \!\!  \!\!\sum_{k_1+k_2+k_3-k=0,\, k>0}\a_{k_1} \a_{k_2} \a_{k_3} \a_{-k} \\
& {}\quad\ \ \  -\i\alpha \!\!\! \!\!\! \!\! \!   \!\!\sum_{k_1+k_2+k_3+k=0, \,k<0} \a_{k_1} \a_{k_2} \a_{k_3} \a_{k} \\
&=0. 
\end{split}
\]
In the last step we made the change of variables $k'=-k$ in the last sum. If we assume uniqueness, we must have $\a_{k} = \bar{\mathsf{a}}_{-k}$ as can be expected because the solution \r{R1} should be real valued. Indeed, take the complex conjugate of \r{9'}, replace $k$ by $-k$; and we get the same system of equations for $\bar{\mathsf{a}}_{-k}$. Then we get the conservation law
\be{R2a}
\sum_{k=1}^\infty   k \big|\a_{k}\big|^2 = \sum_{k=1}^\infty   k \big|\a_{k}\big|_{t=0}\big|^2=\frac{A^2}4. 
\ee
This means that we can look for a solution in what is essentially an $H^{1/2}$ space, see Appendix~\ref{sec_transp}, where we analyze the transport equations. 

\subsection{The inverse problem with real waves} 
We can always assume that the initial condition in \r{R2} is given at $s=0$. 
It is enough to solve \r{9'} with $\alpha=1$. Indeed, let that solution be $\tilde{\mathsf{a}}$; then $\a (s) = \tilde{\mathsf{a}}(\mathcal{R}(s))$ solves \r{9'} with $\mathcal{R}(s) = \int_{0}^s \alpha(\sigma)\,\d\sigma$. In particular, if $\alpha=0$, we get $\a (s) = \tilde{\mathsf{a}}(0)$, which is the vector with only two non-zero components as in \r{R2}. Notice that for $s=T$, the value of $s$ at which the ray has exited $\overline{B(0,R)}$, 
$\mathcal{R}(T)=\int\alpha$ is the X-ray  transform of $\alpha$ in the direction of $\omega$. In fact, the measurements provide $\a_k(T)$ for all $k$. They equal $\tilde a_k(\mathcal{R}(T))$. If this function is invertible for at least one $k$, we recover $\mathcal{R}(T)$, and therefore the X-ray transform of $\alpha$. By the ODE system \r{9'} and  \r{R2a} (or by uniqueness), we cannot have $\d\tilde \a_k/\d s=0$ for all $k$ and a fixed $t$ (recall that $\alpha=1$ for $\tilde a$). Therefore, at least locally, we can invert one of those functions. 
%In typical applications however, one would want to invert $\a_1$. 
We show below that this can be done for small $\alpha$. 

Note that we can view the argument above as a version of the phase shift we encountered  with complex probing waves. The change of variables $s\mapsto \mathcal{R}(s)$ along each characteristic line, having non-negative derivative $\alpha$, deforms $\tilde \a(s)$ to  $\a (s) = \tilde{\mathsf{a}}(\mathcal{R}(s))$. We cannot expect each $\a_k$ to have a constant module as in \r{aoc}, then this change will shift both the phase and the amplitude. 

Since it is not clear if \r{9'} can be solved explicitly with $\alpha=1$, we will linearize the solution near $\alpha=0$. If $\alpha=0$, the solution is $\a ^0=(Ae_{-1}+ Ae_1)/2$ with the standard notation $e_k$ for the vector with all components zero except at the $k$-th position, which is $1$. The  linearization $\delta \a$ at $\alpha=0$ is given by
\[
\delta\a_{k} = -\frac{\i }{2k} \int_0^s \alpha(\sigma) \,\d\sigma \sum_{k_1+k_2+k_3=k} \a_{k_1}^0 \a_{k_2}^0 \a_{k_3}  ^0.
\]
The only non-zero coefficients correspond to $k=-3,-1,1,3$. It is easy to see that
\[
\delta\a_{1}= \delta\bar{\mathsf{a}}_{-1}= -\frac{3A^3}8 \frac{\i }{2} \int_0^s\alpha(\sigma) \,\d\sigma,\quad 
\delta\a_{3}= \delta\bar{\mathsf{a}}_{-3}= -\frac{A^3}8 \frac{\i }{6} \int_0^s \alpha(\sigma) \,\d\sigma.
\]
Therefore, the first order approximation to the r.h.s.\ of \r{R1} without the factor $h^{-1/2}$ about $\alpha=0$ is
\be{R2c}
h^{1/2}  u_1: =A\cos\frac{-t+x\cdot\omega}{h} + \frac{3A^3}8 \mathcal{R} \sin \frac{-t+x\cdot\omega}{h} + \frac{A^3}{24}\mathcal{R} \sin\frac{3(-t+x\cdot\omega)}{h},
\ee
where $\mathcal{R}= \int_{-\infty}^0  \alpha(x+s\omega) \,\d s$, and $A=\chi(-t+x\cdot\omega)$, compare to \r{11}. The first term in \r{R2c} is just the linear solution $u_L$, therefore the linearization $\delta u$ is $\delta u=u_1-u_L$. 

Motivated by this, we take our data to be $\delta u$ at $t=T$, and $x\cdot\omega=T+h\pi/2$ to get $K^3\mathcal{R} /3$ (up to an $O(h)$ error), with $K=\chi(0)$ as before, which  recovers $\mathcal{R}$. For numerical purposes, it is better to integrate $|\delta u|^2|_{t=T}$ over an $h$-independent  range of $x$ values lying away from $\supp\alpha$:
 \be{Data}
\begin{split}
\textrm{Data}^2(z,\omega)&=  h\int_{|s-T|\le\delta} \big| \delta u (T,z+s\omega)\big|^2\,\d s   \\
& = \mathcal{R}^2\int \chi^6(\sigma) \left(  \frac{3^2}{8^2} \sin^2\frac{\sigma}h  + \frac1{24^2}\sin^2\frac{3\sigma}h \right)\d\sigma +O(h^\infty), 
\quad z\in \omega^\perp. 
\end{split}
\ee
Here $\mathcal{R}=X \alpha$ is  the X-ray transform of $\alpha$ in the direction $\omega$. 
The  product $C\sin(\sigma/h) \sin(3\sigma/h)$ is missing because it is $O(h^\infty)$ as it can be seen by writing it as a difference  of cosine functions and integrating by parts. Now, since $\sin^2\beta=(1-\cos(2\beta))/2$, the contribution of each sine squared term above would be $1/2+O(h^\infty)$, hence
\be{R2f}
\begin{split}
\textrm{Data}(z,\omega) = C| X\alpha(z,\omega)|+ O(h^\infty), \quad C:=\frac{\sqrt{41 }  }{24}\Big(\int\chi^6(\sigma)\,\d\sigma\Big)^{1/2},  
\quad z\in \omega^\perp. 
\end{split}
\ee
%We have $\sqrt{41}/24\approx 0.2668$. 
If we filter out the first harmonic in the data only, we need to change the factor $\sqrt{41}/24\approx 0.2668$ in the expression for $C$ in \r{R2f} to $3\sqrt{2}/16\approx 0.2652$.

We formulate this as a theorem now.

\begin{theorem}\label{thm_2}
Let \r{Data} be the integrated measurement of the linearized solution $\delta u$ as $0\le \alpha\to 0$. Then we can recover the X-ray transform $X\alpha$ in the direction $\omega$ by \r{R2f}, as $h\to0$. If we know \r{Data} for all unit $\omega$ in the limit $h\to0$, we can recover $\alpha$ as well. 
\end{theorem}

To connect this approximation with the one in section~\ref{sec_2}, we will disregard for a moment the formation of higher order harmonics, and very formally, consider the $k=-1,1$ ones only. Then the only possible combinations in the sum in \r{9'} are $(-1,1,1)$, $(1,-1,1)$ and $(1,1,-1)$.  Similarly, if $k=-1$ we have those three combinations with the opposite signs. Therefore, the $|k|=1$ harmonics have amplitudes solving
\[
\begin{split}
2\frac{\d}{\d s}a_0^{(1)}+3\i \alpha \left(a_0^{(1)}\right)^2 a_0^{(-1)} &=0,\\
-2\frac{\d}{\d s}a_0^{(-1)}+3\i \alpha \left(a_0^{(-1)}\right)^2 a_0^{(1)}& =0.
\end{split}
\]
 Multiply the first equation by $a_0^{(-1)}$, the second by $a_0^{(1)}$ and subtract them to get $a_0^{(1)}a_0^{(-1)}=\text{const.}$, and by \r{R2}, this constant is $A^2/4$. Therefore, 
 \[
\begin{split}
2\frac{\d}{\d s}a_0^{(1)}+3\i \alpha \frac{A^2}4a_0^{(1)}  &=0,\\
-2\frac{\d}{\d s}a_0^{(-1)}+3\i \alpha \frac{A^2}4 a_0^{(-1)}   & =0.
\end{split}
\]
Using the initial conditions \r{R2}, we get
\[
a_0^{(1)}(s) =\bar a_0^{(1)}(s)= \frac{A}2 \exp\left( -\i\frac{3A^2}8 \int_{-\infty}^s \alpha(\sigma)\,\d \sigma\right).
\]
Therefore, the leading term in \r{R1} corresponding to $k=-1,1$ (the first harmonic) is
\[
u^{(1)}_0 := Ah^{-1/2} \cos\left(\frac{-t+x\cdot\omega}{h} -\frac{3A^2}8 \int_{-\infty}^0 \alpha(x+s\omega)\,\d s\right).
\]
Compare this to \r{R2c}: the first two terms there on the right are well approximated when $0\le \alpha\ll1$ (then  $0\le \mathcal{R}\ll1$ as well). 

 Numerical experiments show that the approximation \r{R2f} is very accurate.

\section{Comparison with other works} We compare our approach with other works of recovery a non-linear term in a semilinear wave equation. \label{sec_comparison}

To put this in a more general context, we notice first that our approach in section~\ref{sec_ps} works if $\Delta$ in \r{1} is the Laplacian associated with some Riemannian metric under some geometric assumptions, for example lack of caustics of the probing waves. Also, the non-linearity could be of different order, like $\alpha u^m$ or $\alpha |u|^{m-1}u$ with $m\ge2$ integer, see also Remark~\ref{rem_g}. On the other hand, a global solution for not necessarily small initial conditions may not exist in some of those more general cases, e.g., for $\alpha$ taking negative values (see also the example in \cite[section~X.13]{Reed-Simon2}) or for $m$ even. The existence of a positive energy functional preserved by the dynamics is essential for the global existence. For general $m$, the weakly non-linear regime happens when $p-1=mp$, see section~\ref{sec_wnl}, i.e., when $p=(1-m)^{-1}$. We are in the linear regime when $p> (1-m)^{-1}$ but to have the non-linearity affecting the terms with relative strength $\sim h$ (rather that $\sim h^\eps$, $0<\eps\ll1$), we need $p\ge 0$, see section~\ref{sec_l}.

The first work, to our knowledge, about recovery of a non-linear term in a semilinear wave equation is \cite{KLU-18}, see also \cite{LassasUW_2016}, and they also recover  the metric. There are several extensions of these works, see, e.g., \cite{Hintz-U-19}. The non-linearity in \cite{LassasUW_2016} is of type $H(x,u)= \alpha_2(x)u^2+\alpha_3(x)u^3+\dots$ in $\R^3$. The main idea there is to collide four ``plane'' waves in general position which in space-time meet at one point. By non-linear interaction, that point acts as a source emitting a spherical wave with a weak singularity (after all, the solution has to be smooth enough for well posedness) of amplitude carrying information about $H$ at that source. More precisely, they have a source $f = \eps_1f_1+ \eps_2 f_2+\eps_3f_3+\eps_4 f_4$ supported in $t\ge0$ appearing in the r.h.s.\ of \r{1}, and zero initial data.  They expand the solution in Taylor series w.r.t.\ $(\eps_1,\eps_2,\eps_3, \eps_4)$  near zero, and show that the information needed to recover $\alpha_2$ is carried out by the (weak) singularities of $\partial_{\eps_1\eps_2 \eps_3\eps_4}u$ at $\eps_1=\eps_2 =\eps_3=\eps_4=0$. Assume for simplicity that all those $\eps_j$'s are of the same size $\eps$. This approach is equivalent to taking two consecutive limits: $\eps\to0$ first (after dividing by $\eps^4$), and $h\to0$ next, if we think about the singular waves as limits of high-frequency ones, i.e., if we think of $u$ or the source $f$ having form of the kind \r{2aa}. The propagation of the waves before the interaction is in the linear regime, in fact they are linear. The useful observed data is $O(\eps^4)$, compared to the  much stronger overall signal $\sim\eps$, and its singularity is weak (required for well-posedness), making it hard to measure. On the other hand, one can create a singularity moving in direction different than the incoming ones. 

The singularities generated  by the interaction of four waves in \cite{KLU-18}, see also \cite{LassasUW_2016},  are not necessarily singularities of the solution of the nonlinear wave equation; they are singularities of a linearization of the nonlinear equation.  For instance the interaction of four conormal waves would produce an open set of singularities of the solution of the wave equation.  The analysis of the  singularities produced by transversal interaction of three semilinear conormal waves  which solve $\square u= f(t,x,u)$ goes back to the work of Bony \cite{Bony1,Bony2} and Melrose and Ritter \cite{MelRit}. Their results essentially say that the new singularities arising from the interaction of three transversal conormal waves is contained on the surface formed by the projection of the bicharacteristics emanating from the set where the three  waves meet.  The principal symbol of the singularities of the new wave was recently computed in \cite{SaWang} in two space dimensions, and in \cite{Antonio2020} in arbitrary dimensions, and it determines $\p_u^3 f(t,x,u(t,x))|_{u=0}$ for all $(t,x)$ on the set where the waves interact. In particular, in the case of a cubic non-linearity, $f(t,x,u)=\alpha(t,x) u^3$,  this gives $\alpha(t,x)$,  for all $(t,x)$ on the set where the three waves interact. By varying  the interacting waves, one can determine $\alpha(t,x)$ on a larger set. 

In \cite{OSSU-principal}, one uses (linear) Gaussian beams which look like \r{2aa} but $\phi$ is complex-valued. Unlike in  the previous works, there is an amplitude $\eps$ and a frequency $1/h$ now. It is shown that three beams are enough in any dimension, as long as the amplitude $\eps$ is small enough. The proof still requires to take the limit $\partial_{\eps_1\eps_2 \eps_3}u$ at  $\eps_1=\eps_2 =\eps_3=0$ \textit{and then} $h\to0$. If we want to keep $\eps$ and $h$ dependent in the spirit of this work, it is enough to take $\eps=1$ (if a global solution exists). Then the Gaussian beams would propagate in the linear regime, and the non-linearity affects the lower order terms only. One can easily see however that the signal generated by the collision (in the resonant case) would be $O(h)$, compared to the overall signal $\sim 1$. This is a direction we do not pursue. 

In this work, we use solutions with amplitude $\eps= h^{-1/2}$, i.e., $\eps$ is not small anymore. They propagate (weakly) non-linearly along their way. We do not collide them but we can think of the non-linear propagation as a continuous self-interaction. The observed signal carrying the useful info is still of the same amplitude, $h^{-1/2}$ which would raise it above the background noise level and makes it comparable to the probing signal. Adding noise to $\Lambda(u_\textrm{in})$ in \r{eq:main_thm}, for example, would multiply its contribution to \r{eq:main_thm} by   $h^{1/2}$ (which will weaken it) and then we take its phase. 

Finally, we want to mention that our PDE \r{1} is scaled for convenience but the physical choice could require multiplying $\Box$ in \r{1} by $h$, see \cite[sec.~4]{Metivier-Notes}. If $v$ solves $h\Box v+\alpha |v|^2v=0$, then $u=h^{-1/2}v$ solves \r{1}; therefore the weakly non-linear solutions correspond to $v\sim 1$, not $\sim h^{-1/2}$.

\appendix

\section{Solving the transport equations} \label{sec_transp}

We start with the leading order equations  \r{9'}. If
\[
\a  (s) = (\ldots, \a_{-k}(s), \ldots, \a_{-2}(s), \a_{-1}(s), \a_0(s), \a_{1}(s), \a_{2}(s),\ldots, \a_{k}(s), \ldots),
\]
we say that
\[
\begin{aligned}
\a (s) &\in l^p, \;\ p\in [1,\infty), & \text{ if } \|\a (s)\|_{l^p}^p& := \sum_{k=-\infty}^\infty |\a_k|^p<\infty, \\
\a (s) &\in l^\infty, \;\  & \text{ if } \|\a (s)\|_{l^\infty}& := \sup_k |\a_k|<\infty,   \\
\a(s) &\in {\mathsf{h}}^m, \;\ m \in [0,\infty),  &\text{ if } \|\a(s)\|_{\mathsf{h}^m}^2& := \sum_{k=-\infty}^\infty |k|^{2m} |\a_k|^2<\infty.
\end{aligned}
\]
When $k=0$, the weight in the definition of $\|\a(s)\|_{\mathsf{h}^m}^2$ vanishes (and it is undefined when $m=0$) which is not the standard choice but our sequences will have non-zero terms when the index is odd only. 
If  $I\subset \mathbb{R}$ is an interval, we say that $\a \in C(I, l^p)$ or $\a \in C(I, {\mathsf{h}}^m)$ if $\a (s)$ is continuous and 
\be{AR2}
\begin{split}
\|\a \|_{C(I, l^p)}= \sup_{s\in I} \|\a (s)\|_{l^p} <\infty, \text{ or } \|\a \|_{C(I, {\mathsf{h}}^m)}= \sup_{s\in I} \|\a (s)\|_{{\mathsf{h}}^m} <\infty.
\end{split}
\ee
The discrete convolution $\a * \mathsf{b}$ is defined to be the sequence such that
\[
(\a * \mathsf{b})_k  = \sum_{k_1+k_2=k}  \a_{k_1} \mathsf{b}_{k_2},
\]
 and in the case of three terms $\a * \mathsf{b}*\mathsf{c}$ is defined such that
\[\
(\a* \mathsf{b}* \mathsf{c})_k = \sum_{k_1+k_2+k_3=k}  \a_{k_1}\mathsf{b}_{k_2} \mathsf{c}_{k_3}.
\]

 Young's inequality holds for discrete convolutions, since it holds for certain groups including $\mathbb{Z}$,  but the standard proof which is an application of H\"older's inequality,  also works in this case, and
 \be{AR4}
 \begin{split}
 & \|\mathsf{a}*\mathsf{b}\|_{l^q}\le \|{\mathsf{a}}\|_{l^{p_1}} \|{\mathsf{b}}\|_{l^{p_2}}, \text{ for }  1+\frac1q= \frac1{p_1}+\frac1{p_2}, \;\ p_1, p_2, q \in [1,\infty], \\
  \|\mathsf{a}*\mathsf{b}*\mathsf{c}\|_{l^q} & \le 
\|{\mathsf{a}}\|_{l^{p_1}} \|{\mathsf{b}}\|_{l^{p_2}} \|{\mathsf{c}}\|_{l^{p_3}}, \text{ for }  1+\frac1q= \frac1{p_1}+\frac1{p_2} +\frac1{p_3}, \;\ p_1, p_2, p_3, q \in [1,\infty]. \\
\end{split}
 \ee
In particular,  if we choose $p_1=p_2=p_3=q=2$ in \r{AR4}, we have the following inequality
\be{AR5}
\|\a *\a *\a \|_{l^2} \le \|\mathsf{a}\|_{l^2}^3.
\ee
This defines the map
\begin{gather*}
T: l^2 \longrightarrow l^2, \quad 
T(\a )= \a *\a *\a .
\end{gather*}

We can write \r{9'}, \r{R2}  as
\be{R7}
\ \a_{k}(s) =  \a_{k}(0)- \frac{\i}{2 k} \int_0^s\alpha(\sigma) (\a *\a *\a )_k(\sigma)\,\d \sigma  , \quad k=\pm1,  \pm2, \dots,
\ee
with $\a_{k}(0)=Ae_{-1}/2+Ae_1/2$. Since only odd values of $k$ will be present, we have  $k\not=0$. 

We want to use Picard iteration and solve \r{R7} locally, if the initial data is  in $l^2,$ and  in view of the conservation law \r{R2a}, show that the solution is in fact global if the initial data is in ${\mathsf{h}}^{\frac12}$.    We define  the map
\[
(\Phi(\mathsf{a}))_k(s)= \mathsf{a}_k(0)- \frac{\i}{2k} \int_0^s\alpha(\sigma) (\a *\a *\a )_k(\sigma)\,\d \sigma  , \quad k=\pm1, \pm 2, \dots.
\]
In view of \r{AR5}, on the interval $[0,s_0]$,  we have
\[
\|\Phi(\a )- \a (0)\|_{C([0,s_0],l^2)}\le \frac{s_0}{2} \|\alpha\|_{L^\infty} \|\a \|_{C([0,s_0],l^2)}^3.
\]
If $\mathcal{H}(\a (0),M)$ denotes the closed ball of radius $M$ centered at $\a (0)$  in the space $C([0,s_0]; l^2)$ 
equipped with the norm \r{AR2}, then as long as $s_0$ is such that 
\be{Cond1}
\frac{1}{2}\|\alpha\|_{L^\infty} (\|\a (0)\|_{l^2}+ M)^3 s_0 \leq M,
\ee
then $\|\Phi(\a )- \a (0)\| \leq M$ and so
\[
\Phi: \mathcal{H}(\a (0),M) \longmapsto \mathcal{H}(\a (0),M).
\]

On the other hand, since
\[
\mathsf{a}*\mathsf{a}*\mathsf{a}- \mathsf{b}*\mathsf{b}*\mathsf{b}= (\mathsf{a}-\mathsf{b})*\mathsf{a}*\mathsf{a}+
(\mathsf{a}-\mathsf{b})*\mathsf{b}*\mathsf{a}+(\mathsf{a}-\mathsf{b})*\mathsf{b}*\mathsf{b},
\]
it follows from \r{AR5} that if $\mathsf{a}, \mathsf{b}\in C([0,s_0]; l^2)$;  then
\[
\begin{split}
\|\Phi(\a )- \Phi(\mathsf{b})\|_{C([0,s_0],l^2)} \le &\\
\frac{s_0}{2} \|\alpha\|_{L^\infty} \|\a -\mathsf{b}\|_{C([0,s_0],l^2)}( \|\a \|_{C([0,s_0],l^2)}^2 & + \|\mathsf{b}\|_{C([0,s_0],l^2)}^2+ \|\mathsf{a}\|_{C([0,s_0],l^2)}\|\mathsf{b}\|_{C([0,s_0],l^2)}).
\end{split}
\]

 And so, if $\mathsf{a}, \mathsf{b}\in \mathcal{H}(\a{}(0),M)$, and if $s_0$ is such that
 \be{Cond2}
 \frac{3s_0}{2} \|\alpha\|_{L^\infty} (M+\|\a (0)\|_{l^2})^2 s_0<1,
 \ee
 the map $\Phi$ is a contraction in $\mathcal{H}(\a (0),M)$,  and therefore \r{R7} has a unique solution 
 $\mathsf{a}(s)\in \mathcal{H}(\a (0),M)\subset C([0,s_0]; l^2)$,  provided $M$ and $s_0$ satisfy \r{Cond1} and \r{Cond2}.  
 
 If in addition we know that the initial data $\a (0)$ is in $\mathsf{h}^{\frac12}$, we will show that we have a unique global solution $\a (s) \in C([0,s_0]; \mathsf{h}^{\frac12})$,  for any $s_0>0$.  To prove this, notice that  once we have a solution 
 $\a (s) \in C([0,s_0]; l^2)$,  then  one can rewrite \r{R7}  as
\[
\sqrt{|k|}\ \a_{k}(s) =  \sqrt{|k|}\ \a_{k}(0)- \sign(k) \frac{\i}{2 \sqrt{|k|}} \int_0^s\alpha(\sigma) (\a *\a *\a )_k(\sigma)\,\d \sigma  , \quad k=\pm 1, \pm 2, \dots,
\]
and in view of \r{AR5},
\[
\| \a \|_{C([0,s_0],\mathsf{h}^{\frac12})} \leq  \| \a (0)\|_{\mathsf{h}^{\frac12}} + \| \a \|_{C([0,s_0],l^2)}^3;
\]
therefore, $\a \in C([0,s_0],h^{\frac12})$,  provided $s_0$ and $M$ satisfy \r{Cond1} and \r{Cond2}. 
Since
\[
\|\a (s)\|_{l^2}\leq  \|\a (s)\|_{\mathsf{h}^{\frac12}},
\]
if we replace the conditions \r{Cond1} and \r{Cond2} with
 \be{Cond3}
 \begin{split}
 \frac{1}{2}\|\alpha\|_{L^\infty} (\|\a (0)\|_{l^2}+ M)^3 s_0 &\le \frac{1}{2}\|\alpha\|_{L^\infty} (M+ \|\a (0)\|_{\mathsf{h}^{\frac12}})^3 s_0 \leq M, \\
 \frac{3}{2} \|\alpha\|_{L^\infty} (M+\|\a (0)\|_{l^2})^2 s_0 &\le \frac{3}{2} \|\alpha\|_{L^\infty} (M+\|\a (0)\|_{\mathsf{h}^{\frac12}})^2 s_0<1,
  \end{split}
  \ee
   then, in view of the conservation law \r{R2a},  
   \[
 \begin{split}
 \frac{1}{2}\|\alpha\|_{L^\infty} (\|\a (s_0)\|_{l^2}+ M)^3 s_0 &\le \frac{1}{2}\|\alpha\|_{L^\infty} (M+ \|\a (s_0)\|_{\mathsf{h}^{\frac12}})^3s_0 \leq M, \\
 \frac{3}{2} \|\alpha\|_{L^\infty} (M+\|\a (s_0)\|_{l^2})^2 s_0&\le \frac{3}{2} \|\alpha\|_{L^\infty} (M+\|\a (s_0)\|_{\mathsf{h}^{\frac12}})^2 s_0<1.
  \end{split}
  \]
This implies that we can solve equation \r{R7} with initial data set at $s=s_0$ instead of $s=0$,  and therefore this shows there exists a unique solution $\a (s)\in C([0,2s_0]; \mathsf{h}^{\frac12})$ to \r{R7}, provided \r{Cond3} is satisfied. This process can be repeated indefinitely and so we have proved the following:

 \begin{proposition}\label{ODE-R} If $\mathsf{a}(0)\in l^{2}$,  then there exists a unique $\a (s)\in C([0,s_0]; l^2)$ which satisfies \r{R7}, as long as $s_0$ and $M$ satisfy \r{Cond1} and \r{Cond2}.   If $\mathsf{a}(0)\in \mathsf{h}^{\frac12}$,  then there exists a unique $\a (s)\in C(\mathbb{R}; \mathsf{h}^{\frac12})$  which satisfies \r{R7}.
 \end{proposition}

We can use induction to obtain a similar result for initial data $\mathsf{a}(0)\in \mathsf{h}^{m},$ with $m\in \mathbb{N}$.   Proposition \ref{ODE-R} guarantees that equation \r{R7} has a unique solution $\a (s)\in C(\mathbb{R}; \mathsf{h}^{\frac12})$ and in particular, $\a (s)\in C(\mathbb{R}; l^2)$.  Moreover, by \r{AR5}, the integrand in \r{R7} is in the same space, therefore $\a_k(s)$ is $C^1$ in the $s$ variable with values in $l^2$, i.e., it is a strong solution of \r{9'}.

Recast \r{9'}  as
\[
\frac{\d}{\d s} k\a_{k}(s) =  -\frac{\i   }{2 k } \alpha(s) k(\a *\a *\a )_k(s) , \quad k=\pm1,  \pm2, \dots,
\]
with some initial condition for  $k\a_k(0)\in \mathsf{h}^{m-1}$. 
Notice that
\[
k(\a*\a*\a)= k\sum_{k_1+k_2+k_3=k} \a_{k_1} \a_{k_2} \a_{k_3}= 3 \sum_{k_1+k_2+k_3=k} \a_{k_1} \a_{k_2} k_3 \a_{k_3}= 3 (\a*\a* k \a).
\]
Therefore, the sequence $\tilde{\a}:=k \a_k$ satisfies
\be{R71}
\frac{\d}{\d s}  \tilde \a_{k}(s) =  -\frac{3\i }{2 k}  \alpha(\sigma) (\a *\a *\tilde\a )_k(\sigma) , \quad k=\pm1,  \pm2, \dots
\ee
This is a linear ODE with a generator
\be{R7'}
(A(s) \mathsf{b} )_k(s)=  - \frac{3\i}{2k}  \alpha(s) (\a *\a *\mathsf{b} )_k(s)  , \quad k=\pm1, \pm 2, \dots.
\ee
By \r{AR5}, $A(s)$ is bounded in $l^2$ with a uniformly bounded norm. Then the solution to \r{R7'} is given by a two-parameter solution group $U(t,s)$ applied to the initial condition  $k\a_k(0)\in l^2$, and solves \r{R7'} in strong sense, see \cite[Theorem~X.96]{Reed-Simon2}. That group is obtained by successive iterations of the integrated equation (the Dyson expansion), similarly to what we did above  for the non-linear equation. 

  So we have shown that
  if $\mathsf{a}(0)\in \mathsf{h}^{1},$  then there exists a unique $\a (s)\in C(\mathbb{R}; \mathsf{h}^{1})$ which satisfies \r{R7}.
  
 When $m=2,$ we have
 \[
 \begin{split}
 k^2 (\a*\a*\a)_k &= k^2\sum_{k_1+k_2+k_3=k} a_{k_1} a_{k_2} a_{k_3}=  3k \sum_{k_1+k_2+k_3=k} a_{k_1} a_{k_2} k_3 a_{k_3} \\  & =  6 \sum_{k_1+k_2+k_3=k} a_{k_1} k_2a_{k_2} k_3a_{k_3}+ 3 \sum_{k_1+k_2+k_3=k} a_{k_1} a_{k_2} k_3^2 a_{k_3}.
 \end{split}
 \]
Note that this equality can be interpreted as taking second derivative w.r.t.\ the dual variable of $k$, when the convolution is a product; and this is how we got \r{9'} in the first place, see \r{R1}.

Therefore, the ODE for $k^2\a_k$ is
\be{R72}
\frac{\d}{\d s}  k^2 \a_{k}(s) =- \frac{3\i }{2 k}  \alpha(\sigma) (\a *\a *k^2\a )_k(s)  
 - \frac{6\i }{2 k}  (\a *k \a *k \a )_k(s) , \quad k=\pm1,  \pm2, \dots.
\ee
This is a linear non-homogeneous ODE, similar to \r{R71} (which is homogeneous) and can be solved in $l^2$ used the solution group $U(t,s)$ and Duhamel's principle. The generator is the same as above, and the source terms is continuous in $s$ with values in $l^2$ by the previous step.

We already know that $\a \in C(\mathbb{R}, \mathsf{h}^{1})$ and so the right hand side of \r{R72} is in $C(\mathbb{R};l^2),$ and so the same contraction mapping argument, now applied to the non-homogeneous equation, shows that  if $\mathsf{a}(0)\in \mathsf{h}^{2},$  then there exists a unique $\a (s)\in C(\mathbb{R}; \mathsf{h}^{2})$ which satisfies \r{R7}.

In general if $m$ is a positive integer, $k^m \a$ satisfies
\be{R7m}
\begin{split}
\frac{\d}{\d s} k^m \a_{k}(s)  &=- \frac{3\i }{2 k}  \alpha(\sigma) (\a *\a *|k|^m\a )_k(\sigma) \\
 &\qquad + \frac{\i }{ 2 k} \sum_{\stackrel{m_1+m_2+m_3=m}{m_1>0}}  C_{m_1,m_2,m_3}  \alpha(\sigma) (|k|^{m_1}\a *|k|^{m_2}\a *|k|^{m_3}\a )_k(\sigma) , 
\end{split}
\ee
$k=\pm1,  \pm2, \dots$. 
By induction, the right hand side is in $C(\mathbb{R}, l^2)$ and the argument used above proves the following. 
\begin{proposition}\label{ODE-Rm}   If $m$ is a positive integer and $\mathsf{a}(0)\in \mathsf{h}^{m},$  then there exists a unique $\a (s)\in C(\mathbb{R}; \mathsf{h}^{m})$ which satisfies \r{9'} with that initial condition. 
 \end{proposition}

Next we analyze the higher order transport equations.   We return to the notation $a_0^{(k)}$ instead of $\a_k$. 
By \r{exp}, the next transport equation takes the form
\be{tr-h1}
2k \frac{\d}{\d s} a_1^{(k)} +3\i (a_0* a_0* a_1)^{(k)}= -\i \Box a_0^{(k)}
\ee
with zero initial conditions. Note that the D'Alembertian in the r.h.s.\ is written in the original $(t,x)$ coordinates instead in the characteristic coordinates $(s,y)=( t,x-t\omega)$ but we can always convert it to the latter ones. 
This is a {linear} homogeneous system of ODEs, but before we can solve it we need to show that the right hand side is well defined. We can differentiate \r{R7} to find 
\[
\ \partial_{y_j}\a_{k}(s) =  \partial_{y_j}\a_{k}(0)- \frac{3\i}{2 k} \int_0^s\alpha(\sigma) (\a *\a *\partial_{y_j}\a )_k(\sigma)\,\d \sigma  , \quad k=\pm1,  \pm2, \dots,
\]
and the argument used to prove Proposition \ref{ODE-Rm} shows that $\p_{y_j}\a (s)\in C(\mathbb{R}; \mathsf{h}^{m}),$  provided 
$\p_{y_j}\a (0)\in  \mathsf{h}^{m},$ and $m$ is a positive derivative.  For the  second order derivatives, we have
\be{syst1}
\begin{split}
\  \partial_{y_j}\p_{y_r}\a_{k}(s) &=  \partial_{y_j}\p_{y_r}\a_{k}(0)- \frac{3\i}{2 k} \int_0^s\alpha(\sigma) (\a *\a *\p_{y_j}\partial_{y_r}\a )_k(\sigma)\,\d \sigma  \\
&\qquad  -\frac{9\i}{2 k} \int_0^s\alpha(\sigma) (\a *\p_{y_m}\a *\partial_{y_j}\a )_k(\sigma)\,\d \sigma  , \quad k=\pm1,  \pm2, \dots.
\end{split}
\ee
We have already established that the first order derivatives satisfy 
\[
 \int_0^s\alpha(\sigma) (\a *\p_{y_m}\a *\partial_{y_j}\a )_k(\sigma)\,\d \sigma  \in C(\mathbb{R}; \mathsf{h}^m), 
\]
 and again we apply the argument used in the proof of Proposition  \ref{ODE-Rm} to the non-homogeneous system \r{syst1} and we find that $\p_{y_j}\partial_{y_r}\a \in C(\mathbb{R}; \mathsf{h}^m)$.  We can treat derivatives involving $\partial_s$ in a similar way, differentiating \r{R7} w.r.t.\ $s$. 
 Once we obtain the result for second order derivatives, the same argument proves the result for third order derivatives and so by induction we obtain the following.

\begin{proposition}\label{ODE-DER}   If $m$ is a positive integer and $(\partial_s, \partial_y)^\alpha \mathsf{a}(0)\in \mathsf{h}^{m},$  for $|\alpha|\leq M,$ then there exists a unique $\a (s)\in C(\mathbb{R}; \mathsf{h}^{m})$ such that $(\partial_s, \partial_y)^\alpha \a (s)\in C(\mathbb{R}; \mathsf{h}^{m})$  for $|\alpha|\leq M,$ which satisfies \r{R7}.
 \end{proposition}

In particular, if  $(\partial_s,\partial_y)^\alpha \mathsf{a}(0)\in \mathsf{h}^{m},$ with $|\alpha|\leq 2$ it follows that
$\Box a_0^{(k)} \in C(\mathbb{R}; \mathsf{h}^m),$ and we can once again apply the argument used in the proof of Proposition \ref{ODE-Rm} to show the following. 
\begin{proposition}\label{ODE-A1}   If $m$ is a positive integer and $(\partial_s,\partial_y)^\alpha \mathsf{a}(0)\in \mathsf{h}^{m},$ with $|\alpha|\leq 2,$  then there exists a unique $a_1 (s)\in C(\mathbb{R}; \mathsf{h}^{m})$  which satisfies \r{tr-h1}.
 \end{proposition}

The higher order equations can be treated similarly.

\section{Global existence of solutions and well-posedness}  \label{sec_gl}
We formulate global existence and well-posedness results for the semilinear wave equation \r{1} when $n=2,3$, with initial conditions
\be{G1}
u|_{t=0} = f_1, \quad u_t|_{t=0} = f_2. 
\ee
Related results can be found in  \cite{dodson2018global,  Jorgens1961, Segal63, Ebihara72, Heinz-Wahl1975, Brenner79}. We follow \cite[section~X.13]{Reed-Simon2}, where even more general non-linearities are considered. The theorems below follow from the theorems there, see more specifically pp.~303--310, when $n=3$. We will show that they hold when $n=2$ as well. We will modify the energy space a bit. We are interested in solutions with initial data belonging to the energy space locally only propagating over time interval $[0,T]$ with $R>0$, $T>0$ fixed.   By \cite[Theorem~X.77]{Reed-Simon2}, the speed of propagation does not exceed one. Then it is enough to study initial conditions supported in the ball $B(0,R)$, see the paragraphs following Theorem~\ref{thm_ex}. 
The support of the solution would not expand beyond $B(0,R+T)$. We can just work in the latter ball by imposing zero boundary conditions on its boundary. The solutions we are interested in would never reflect from the boundary. In what follows, we replace $R+T$ by $R$. The energy space then becomes $\mathcal{H} := H_0^1(\Omega)\times L^2(\Omega)$, where $\Omega=B(0,R)$. Then $-\Delta$ is essentially self-adjoint on $C_0^\infty(\Omega)$, extending the Dirichlet Laplacian $-\Delta_D$ on $\Omega$ as a self-adjoint one, having a positive minimal eigenvalue. Then $B=(-\Delta_D)^{1/2}$ is a well defined positive operator on $L^2(\Omega)$. Moreover, $D(B)=H_0^1(\Omega),$ and for every $f\in D(B)$, we have $\|Bf\|=\|\nabla f\|$. 

In \cite[section~X.13]{Reed-Simon2}, there is the Klein-Gordon term $m^2$ added to the Laplacian with $m>0$, then $B=(-\Delta+m^2)^{1/2}$ in $L^2(\R^3)$. All the proofs apply to our situation as well. Another way to make the mass $m=0$ is outlined in Problem~76 there: add $m^2$ to $-\Delta$ and subtract it from the non-linearity $\alpha |u|^2u$. The space dimension is $n=3$ there however. 

The well-posedness for $n=2,3$ also follows from \cite{Heinz-Wahl1975} in a similar way. They consider more general non-linearities as well. 

\subsection{Existence and uniqueness} 
We view \r{1} as on ODE in the energy space $\dot H^1(\Omega) \times L^2(\Omega)$, as it is usually done:
\be{e10nn}
 \mathbf{u}_t = \begin{pmatrix} 0&\Id \\ \Delta &0  \end{pmatrix}\mathbf{u} -\begin{pmatrix} 0\\-\alpha |u|^2u \end{pmatrix}, \quad \mathbf{u}(0)=\mathbf{f}:= \begin{pmatrix} f_1\\ f_2 \end{pmatrix}.
\ee
The space $\dot H^1(\Omega)$ is defined as the completion of $C_0^\infty(\Omega)$ under the norm $\|f\|_{\dot H^1}=\|\nabla f\|_{L^2}$. If $\Omega$ is a bounded domain, then $\dot H^1(\Omega)$ is topologically equivalent to $H_0^1(\Omega)$. We denote by $A$ the matrix operator above. Its domain is $D(A) =  H^2(\Omega) \cap H_0^1(\Omega)\times H_0^1(\Omega)$.

One uses the Picard iteration to solve it. We convert it to an integral equation, i.e., we are seeking the weak solution now:
\be{e11nn}
\mathbf{u}(t) = U_0(t)\begin{pmatrix} f_1\\ f_2 \end{pmatrix} - \int_0^t U_0(t-s) \begin{pmatrix} 0\\ \alpha |u(s)|^2u(s)\end{pmatrix}\d s,
\ee
where $U_0$ is the solution group of the linear equation and we suppressed the dependence on $x$ in $u$. Then we replace $u$ in the non-linearity on the right with $\mathbf{u}_0(t) :=U_0(t)\mathbf f$, compute the first iteration by that formula, then iterate, and take the limit. For this to work at least locally, we need the non-linearity to map $\dot H^1$ to $L^2$  continuously and be Lipschitz there. That would give us a (weak) solution in the energy space only. For a strong solution, one needs the Cauchy data to be in the domain $D(A)$ of the matrix operator in \r{e10nn}, and wants to prove that the solution exists in that space as well. The analysis is similar but in a new space. To prove existence of a global solution, the energy preservation \r{1a} plays a crucial role. 
This is the strategy in \cite[section~X.13]{Reed-Simon2}, as well as in the papers cited above.

We assume $n=2$ or $n=3$ below. We will show that the sequence of lemmas in \cite[section~X.13]{Reed-Simon2} which imply the desired theorem hold when $n=2$ as well but we also allow $n=3$  below. All norms are in $\Omega\subset B(0,R)$.

The first lemma shows that the non-linearity is a continuous operator in the energy space, see also Lemma~\ref{lemma6.2} below.
\begin{lemma}\label{lemma6.1}
For every $u\in C_0^\infty(\Omega)$, we have
\be{e5nn}
\|u\|_{L^6} \le    C \|\nabla u\|_{L^2}.
\ee
\end{lemma}
\begin{proof}
By the Sobolev embedding inequality, 
\be{e1nn}
\|u\|_{L^q}\le C(n,p)\|\nabla u\|_{L^p}, \quad 1\le p< n, \quad 1/q =1/p-1/n.
\ee
Set $n=3$, $p=2$ in \r{e1nn}, then $q=6$; hence 
\be{e5nn6}
 \|u\|_{L^6}\le C \|\nabla u\|_{L^2}, \quad n=3.
\ee
In \cite{Heinz-Wahl1975} one can find a refined argument which covers $n=2$ as well, and is also useful to prove the Lipschitz property below for $n=2,3$. 
 Writing $|u|^6= |u|^2 |u|^4$, we apply H\"older's inequality first
\[
\int \left|\alpha |u|^2u\right|^2 \,\d x  \le \|u\|^4_{L^{4q_1}}\|u\|^2_{L^{2q_2}} ,\quad  1/q_1+1/q_2=1,\quad  q_1, q_2>0.
\]
Take $4q_1=2q_2$; then $q_1=3/2$, $q_2=3$. 
We apply \r{e1nn} with $n=2$, $q=4q_1=2q_2=6$. For the corresponding $p$, we get $p_1=p_2=6n/(6+n)$, which does not exceed $2$ when $n=2,3$.  Another application of H\"older's inequality to the pair of functions $u$ and $1$ implies that \r{e5nn6} still holds for $n=2$ as well. Note that the constant $C$ in \r{e5nn} is independent of $R$ when $n=3$ but it depends on it when $n=2$.  
\end{proof}

Next lemma is a refinement of the previous one. 
\begin{lemma}\label{lemma6.2}
For every $u\in H_0^1(\Omega)$, we have
\[
\|u_1u_2u_3\|_{L^2} \le    C \|\nabla u_1\|_{L^2}\|\nabla u_2\|_{L^2}\|\nabla u_3\|_{L^2}.
\]
\end{lemma}

The proof is as in \cite[Lemma~3, X.13]{Reed-Simon2}; in particular, Lemma~\ref{lemma6.1} above is used. 

The next lemma says that non-linearity is a continuous operator in the energy space as it follows directly from Lemma~\ref{lemma6.1}, and that it is Lipschitz there. 
\begin{lemma}\label{lemma6.3}
For every $u_1, u_2\in H_0^1(\Omega)$, we have
\be{eRS1}
\begin{split}
\left\|\alpha |u_1|^2 u_1\right\|_{L^2} &\le    C \|\nabla u_1\|^3_{L^2},\\
\left\| \alpha |u_1|^2 u_1 - \alpha |u_2|^2 u_2\right\|_{L^2} &\le C(u_1,u_2)\|\nabla(u_1-u_2)\|_{L^2}
\end{split}
\ee
with $ C(u_1,u_2) = C_0\left(\|\nabla u_1\|^2 +\|\nabla u_1\|\|\nabla u_2\| +\|\nabla u_2\|^2\right)$ and $C_0>0$ depending on $R$ only. 
\end{lemma}

The proof is as in \cite[Lemma~4, X.13]{Reed-Simon2} and it is based on the previous lemmas. In particular, it works for $n=2$ as well as in all lemmas so far. 

Next lemma is an analogue of Lemma~\ref{lemma6.3} but the smoothness requirements are one degree higher, so are the conclusions. It corresponds to \cite[Lemma~5, X.13]{Reed-Simon2}.

\begin{lemma}\label{lemma6.4}
For every $u_1, u_2\in  H^2(\Omega)\cap H_0^1(\Omega)$, we have
\[
\begin{split}
\left\|\nabla(\alpha |u_1|^2 u_1)\right\|_{L^2} &\le    C \|\nabla u_1\|_{L^2}^2 \|\Delta  u_1\|_{L^2},\\
\left\| \nabla\left(\alpha |u_1|^2 u_1 - \alpha |u_2|^2 u_2\right) \right\|_{L^2} &\le C(\nabla u_1,\nabla u_2, \Delta u_1,\Delta u_2 )\|\Delta(u_1-u_2)\|_{L^2}
\end{split}
\]
with $ C$ above some continuous function, increasing in each of its arguments. 
\end{lemma}

\begin{proof}[Sketch of the proof]
The proof is the same as that of \cite[Lemma~5, X.13]{Reed-Simon2} with one caveat. The latter uses the Fourier transform since $\Omega=\R^3$ there. We can adapt this to the current setup however. In  the proof of \cite[Lemma~5, X.13]{Reed-Simon2}, one needs to estimate $\|\nabla u_{x_j}\|$. For every $u$ as in the lemma, we have
\[
\|\nabla u_{x_j}\|_{L^2}\le C\|\Delta u\|_{L^2}
\]
by standard elliptic estimates about the solution to $\Delta u=f$, $u|_{\bo}=0$. This is also the estimate first established in the proof of \cite[Lemma~5, X.13]{Reed-Simon2} using the Fourier transform. Another inequality used there is $\|\nabla u\|_{L^2}\le C \|\Delta u\|_{L^2}$ for $u$ as in the lemma, which follows again from standard elliptic estimates. The rest of the proof is the same as in \cite[Lemma~5, X.13]{Reed-Simon2}. 
\end{proof}

The next lemma states the energy preservation property \r{1a} for solutions with regularity as in Lemma~\ref{lemma6.4} above. Of course,  assuming enough smoothness, \r{1a} is immediate.
\begin{lemma}\label{lemma6.5}
Let $u$ be a solution of \r{1}, \r{G1} on $[0,T)$ with $\mathbf{f} = (f_1,f_2) \in D(A)$. Then the energy \r{1a} is independent of $t$.
\end{lemma}

Let $E_0(\mathbf{u}(t))$ be  the ``free'' energy, defined as $E(\mathbf{u}(t))$ in \r{1a} but with $\alpha=0$, i.e., $E_0(\mathbf{u}(t))=\frac12\|\mathbf{u} (t)\|^2_\mathcal{H}$.  In the next lemma, $\Omega=\R^n$ instead of being bounded because this case is of its own interest as $T$ grows and the support expands. The proof applies when $\Omega=B(0,R)$ as well; then \r{e2n} still holds, and the rest is unchanged.

\begin{lemma}\label{lemma_en} 
Assume that $\mathbf{f} = (f_1,f_2)\in D(A)$   and is supported in the ball $B(0,R)$, and let  $\mathbf{u} =(u,u_t)$ solve \r{1}, \r{G1}. Then $E_0(\mathbf{u}(t))$ remains bounded with a bound dependent on $R$ but independent of $T$.
\end{lemma}

\begin{proof}
Recall Poincar\'e's inequality
\be{e2n}
\|f\|_{L^2}\le CR \|\nabla f\|_{L^2}, \quad \supp f\subset B(0,R). 
\ee
In particular,
\[
\|f\|_{H^1}\le C(1+R)\|\nabla f\|_{L^2}, \quad \supp f\subset B(0,R). 
\]
By  the H\"older inequality, if $v$ is supported in a bounded domain $\Omega\subset B(0,R)$, 
\[
\begin{split}
\int_\Omega |v|^p\,\d x &\le  \Big(\int_\Omega  |v|^{p q'}\,\d x\Big)^\frac{1}{q'} \Big(\int_\Omega   \,\d x\Big)^\frac{1}{q''} \\
&\le  CR^{n/q''}  \Big(\int_\Omega  |v|^{p q'}\,\d x\Big)^\frac{1}{q'}, \qquad \frac1{q'}+\frac1{q''}=1,\quad  q'>1, q''>1,
\end{split}
\]
thus 
\be{e0n}
\|v\|_{L^p}\le C R^{n/p-n/q} \|v\|_{L^q}\quad \text{as long as $p\le q$.}
\ee 
We apply the Sobolev embedding inequality \r{e1nn} with   $q=4$. Then for the non-quadratic term in the definition \r{1a} of $E(\mathbf{u}(t))$ we have
\[
\int \alpha |u|^4\,\d x \le C \|\nabla u\|^4_{L^p}, \quad 1/4=1/p-1/n\quad  \Longrightarrow \quad p=\frac{4n}{n+4}.
\]
When  $n\le4$, we have $p\le2$. 
Therefore, for every $\mathbf{u}(0)$ supported in $B(0,R)$, we have, for $t=0$,
\[
\int \alpha|u|^4\,\d x \le C R^{4} \|\nabla u\|^4_{L^2}\le C R^{4} E_0^2 (\mathbf{u}(0)), 
\]
where we used \r{e0n}. In fact, the first inequality is a known generalized version of the Poincar\'e inequality. 
Hence,
\[
E(\mathbf{u}(0))\le C\left( R^{4} E_0^2(\mathbf{u}(0)) + E_0(\mathbf{u}(0)) \right).
\]
Since the energy is preserved,
\[
 E_0(\mathbf{u}(t))\le  E(\mathbf{u}(t))=  E(\mathbf{u}(0)) \le  C\big( E_0(\mathbf{u}(0)) +  R^{4} E_0^2(\mathbf{u}(0)) \big).
\]
\end{proof}

The analysis in \cite[section~X.13]{Reed-Simon2}, see Theorem~X.75 there, yields the following. 

\begin{theorem}\label{thm_ex} Let $n=2$ or $n=3$. Let $\alpha\in C^2(\R^n)$. 
Assume that $(f_1,f_2)\in H^2(\R^n)\cap \dot H^1(\R^n)\times \dot H^1(\R^n)$ is compactly  supported. Then the PDE \r{1} with Cauchy data \r{G1} has a unique solution in $\R_t\times\R^n_x$ so that $u\in C^j(\R_t;\; H^{2-j}(\R^n) )$, $j=0,1,2$.
\end{theorem}

In fact, we first prove the theorem with $\R^n$ replaced by a bounded domain $\Omega$. Then using the finite speed of propagation, we reduce the $\Omega=\R^n$ case to this one, as explained at the beginning of this section. 

\begin{remark}\label{remark_B1}
Assume that $\alpha$ has compact support. Then we can remove the requirement that $\mathbf{u}(0)=(u_1,u_2)$ has compact support and the latter needs to belong to the indicated space locally only. We can localize $\mathbf{u}(0)$ in a large ball with some smooth cutoff $\chi$ so that signals supported outside it do not reach $\supp\alpha$ for time $T>0$ fixed (they solve the linear wave equation there). Write $\mathbf{u}(0)=\chi \mathbf{u}(0)+ (1-\chi)\mathbf{u}(0)$. Apply the theorem to solutions with initial data the first term; and solve the linear problem with initial data the second one getting a solution with finite local energy. Then the sum solves the non-linear problem for $|t|\le T$. 
\end{remark}

\subsection{Well-posedness}
The following theorem and its proof correspond to \cite[Theorem~X.75]{Reed-Simon2}. Problem~80 there shows that one can increase the Sobolev norms in which the estimates are made, i.e., work in $D(A^k)$ with $k\ge2$.

\begin{theorem}\label{thm_stab} Let $n=2$ or $n=3$ and $\alpha\in C^2(\Omega)$. 
Let $u^{(1)}$, $u^{(2)}$ solve
\[
\begin{split}
u_{tt}^{(j)}-\Delta u^{(j)} + \alpha(x) |u^{(j)}|^2u^{(j)}& =0,\\
u^{(j)}|_{t=0}&=f_1^{(j)}, \\
u_t^{(j)}|_{t=0}&=f_2^{(j)},
\end{split}
\]
$j=1,2$. Assume $\|\mathbf{f}^{(j)}\|_{\mathcal{H}}\le C_0$, $j=1,2$, where  $\mathbf{f}^{(j)}= (f_1^{(j)}, f_2^{(j)})$.  Then
\be{e10}
\| \mathbf{u}^{(1)}(t)- \mathbf{u}^{(2)}(t)\|_{D(A)}\le e^{C(C_0)t}  \| \mathbf{f}^{(1)} -\mathbf{f}^{(2)} \|_{D(A)}.
\ee
\end{theorem}

\begin{proof}
Dropping the superscripts, we have, similarly to \r{e11nn},
\be{e12}
\mathbf{u}(t) = U_0(t)\mathbf{u}_0 - \int_0^t U_0(t-s) \begin{pmatrix} 0\\ \alpha |u(s)|^2u(s)\end{pmatrix}\d s .
\ee
Subtract those identities for $j=1,2$ and use Lemma~\ref{lemma6.3} to get
\[
\begin{split}
\| \mathbf{u}^{(1)}(t)- \mathbf{u}^{(2)}(t)\|_{\mathcal{H}}&\le \| \mathbf{f}^{(1)} -\mathbf{f}^{(2)} \|_{\mathcal{H}}  
+ C(C_0)  \int_0^t\| \mathbf{u}^{(1)}(s)- \mathbf{u}^{(2)}(s) \|_\mathcal{H},
\end{split}
\]
where $C(C_0) $ is the constant in the second inequality in \r{eRS1} which depends on $C_0$ only since the free energy $E_0(\mathbf{u}(t))$ remains bounded by Lemma~\ref{lemma_en}. Then  \r{e10}, with the norms there in $\mathcal{H}$ instead of $D(A)$,   follows by Gronwall's inequality.

To prove  \r{e10} with the norms as stated, we need an a priori bound for $E_0(A\mathbf{u}(t)) =\frac12\|A\mathbf{u}(t)\|^2_\mathcal{H} $. This follows from \cite[Lemma~1]{Reed-Simon2}, the needed conditions for it to hold in our situation are guaranteed by Lemma~\ref{lemma6.3}. Then we apply $A$ to \r{e12} and argue as above. 
\end{proof}

\begin{remark}
As we mentioned above, Problem~80 in \cite{Reed-Simon2} outlines a way to prove even higher order estimates. For that, one needs to prove higher order versions of Lemmas~\ref{lemma6.3}--\ref{lemma6.5}, see \cite[Theorem~X.74]{Reed-Simon2}. 
\end{remark}

The next theorem shows that the formal asymptotic solution (parametrix) is close to an actual one; thus justifying the parametrix construction.

\begin{theorem}\label{thm_stab2} Let $n=2$ or $n=3$ and $\alpha\in C^2(\Omega)$. 
Let $u$ solve  the unperturbed equation \r{1} with initial conditions \r{G1},  
where $\|\mathbf{f}\|_{\mathcal{H}}\le C_0$.  
Let $u^\sharp$ solve 
\[
\begin{split}
u_{tt}^\sharp -\Delta u^\sharp  + \alpha(x) |u^\sharp|^2u^\sharp & =r(t,x),\\
u ^\sharp|_{t=0}&=f_1^\sharp, \\
u_t^\sharp |_{t=0}&=f_2^\sharp,
\end{split}
\]
with $\|\mathbf{u}^\sharp(t)\|_{D(A)}\le C^\sharp \|\mathbf{f}\|_{D(A)}$ for $t\in [0,T]$.  
 Then
\[
\| \mathbf{u}(t)- \mathbf{u}^\sharp(t)\|_{D(A)}\le e^{C(C_0, C^\sharp)t}\left(  \int_0^t\| r(s,\cdot)  \|_{H^1}\, \d s +   \| \mathbf{f}  -\mathbf{f}^\sharp\|_{D(A)}\right)
\]
 for $t\in [0,T]$ assuming that  $r$ is such that the norm in the r.h.s.\ above for each one of them is finite. 
\end{theorem}

\begin{proof}
We argue as above. Subtracting the two solutions, we get
\[
\begin{split}
\| \mathbf{u} (t)- \mathbf{u}^\sharp(t)\|_{\mathcal{H}}&\le \| \mathbf{f} -\mathbf{f}^\sharp \|_{\mathcal{H}} + C(C_0, C^\sharp)  \int_0^t\| \mathbf{u}(s)- \mathbf{u}^\sharp(s) \|_\mathcal{H}\, \d s %\\ &\qquad 
+    \int_0^t\| r(s, \cdot) \|_{L^2}\,\d s.
\end{split}
\]
Now we apply $A$ to the difference and estimate in $\mathcal{H}$ again. Note that the needed a priori estimate for $\|\mathbf{u}\|_{D(A)}$ is guaranteed by the argument in the previous proof while that for $\|\mathbf{u}^\sharp\|_{D(A)}$ is postulated but it naturally holds for the parametrix we constructed.  
\end{proof}

\begin{remark} \label{rem_final}
By Sobolev embedding, since $n=2$ of $n=3$, for every $\mathbf{u}=(u,u_t)\in D(A)$, we have $u\in H^2$ for every $t$, therefore, $u\in C^0$ as well, with a continuous dependence on $t$. The estimates above hold in $C^0$ as well for $u(t,\cdot)$. 
\end{remark}

%\bibliographystyle{abbrv}
%\bibliography{../myreferences}

\end{document}